\documentclass{amsart}
\usepackage{amsmath,amsfonts,amssymb,amscd,amsthm,color}
\usepackage{hyperref}
\usepackage{graphicx}
\usepackage{amsmath,amsfonts,amssymb,amscd,amsthm,color}
\usepackage{rotating}
\usepackage{vmargin}
\newtheorem{theorem}[equation]{Theorem}
\newtheorem{lemma}[equation]{Lemma}
\newtheorem{proposition}[equation]{Proposition}
\newtheorem{definition}[equation]{Definition}
\newtheorem{corollary}[equation]{Corollary}
\newtheorem{scholium}[equation]{Scholium}
\newtheorem*{property}{Property ($*$)}
\theoremstyle{remark}
\newtheorem{remark}[equation]{Remark}
\numberwithin{equation}{section}

\newcommand\ad{{\text{ad}}}
\newcommand\sco{{\text{sc}}}
\DeclareMathOperator\codim{codim}
\DeclareMathOperator\End{End}
\DeclareMathOperator\Gl{\text{GL}}
\DeclareMathOperator\IC{IC}
\DeclareMathOperator\Id{Id}
\DeclareMathOperator\Ind{Ind}
\DeclareMathOperator\Irr{Irr}
\DeclareMathOperator\Proj{Proj}
\DeclareMathOperator\PSO{\text{PSO}}
\DeclareMathOperator\PSp{\text{PSp}}
\DeclareMathOperator\Orth{\text{O}}
\DeclareMathOperator\rank{rank}
\DeclareMathOperator\reg{\text{reg}}
\DeclareMathOperator\Res{Res}
\DeclareMathOperator\Sh{\text{Sh}}
\DeclareMathOperator\SL{\text{SL}}
\DeclareMathOperator\SO{\text{SO}}
\DeclareMathOperator\Spin{\text{Spin}}
\DeclareMathOperator\Sp{\text{Sp}}
\DeclareMathOperator\supp{supp}
\DeclareMathOperator\Sym{Sym}
\DeclareMathOperator\wf{wf}

\newcommand\BF{{\mathbb F}}
\newcommand\BQ{{\mathbb Q}}
\newcommand\BZ{{\mathbb Z}}
\newcommand\BN{{\mathbb N}}

\newcommand\bG{{\bf G}}
\newcommand\bL{{\bf L}}
\newcommand\bT{{\bf T}}
\newcommand\ci{{\iota}}
\newcommand\cj{{\gamma}}
\newcommand\ck{{\kappa}}
\newcommand\cl{{\lambda}}

\newcommand\Sgot{{\mathfrak S}}
\newcommand\CA{{\mathcal A}}
\newcommand\CC{{\mathcal C}}
\newcommand\CE{{\mathcal E}}
\newcommand\CF{{\mathcal F}}
\newcommand\CG{{\mathcal G}}
\newcommand\CI{{\mathcal I}}
\newcommand\CL{{\mathcal L}}
\newcommand\CM{{\mathcal M}}
\newcommand\CP{{\mathcal P}}
\newcommand\CS{{\mathcal S}}
\newcommand\CT{{\mathcal T}}
\newcommand\CX{{\mathcal X}}
\newcommand\CY{{\mathcal Y}}
\newcommand\CZ{{\mathcal Z}}
\newcommand\GF{{\bG^F}}
\newcommand\WGL{{W_\bG(\bL)}}
\newcommand\Qlbar{{\overline\BQ}_\ell}
\newcommand\Fq{{\BF_q}}
\newcommand\Fqbar{{\overline\BF}_q}
\newcommand{\inv}{^{-1}}
\newcommand{\uni}{{\text{uni}}}
\newcommand\GFuni{{\bG^F_\uni}}
\newcommand{\lf}{Lusztig function}
\renewcommand{\ggg}{generalised Gelfand-Graev character}
\newcommand\wavefront{wave front set}
\newcommand{\cfs}{\hat\bG^F}
\newcommand{\cf}{{\it cf.}}
\newcommand{\ie}{{\it i.e.}}
\newcommand{\be}{\begin{equation}}
\newcommand{\ee}{\end{equation}}
\newcommand{\lr}{\longrightarrow}
\newcommand{\wt}{\widetilde}
\newcommand{\ol}{\overline}
\newcommand{\ot}{\otimes}
\newcommand{\ve}{\varepsilon}
\newcommand{\vp}{\varphi}
\newcommand{\scal}[3]{{\langle\,#1,#2\,\rangle_{#3}}}
\newcommand\lexp[2]{\kern\scriptspace\vphantom{#2}^{#1}\kern-\scriptspace#2}
\makeatletter
\@namedef{subjclassname@2010}{%
  \textup{2010} Mathematics Subject Classification}
\makeatother

\begin{document}

\renewcommand{\thefootnote}{\fnsymbol{footnote}}

\footnotetext{Received March 13, 2014.}


\title[On Character Sheaves and Characters of Reductive Groups...]{On Character Sheaves and Characters of Reductive Groups at Unipotent Classes}

\author[Fran\c cois~Digne, Gustav~Lehrer, and Jean~Michel]{Fran\c cois~Digne, Gustav~Lehrer, and Jean~Michel}

\begin{abstract}
\noindent{\bf Abstract:} With a view to determining character values of finite reductive groups at unipotent
elements, we prove a number of results concerning inner products of \ggg s
with characteristic functions of character sheaves, here called \lf s.
These are used to determine projections of \ggg s
to the space of unipotent characters, and to the space of characters with a given \wavefront.
Such projections are expressed largely in terms of Weyl group data.
We show how the values of characters at their unipotent support or \wavefront\
are determined by such data. In some exceptional groups we show that the projection
of a \ggg\ to a family with the same \wavefront\ is (up to sign) the dual of a Mellin transform.
Using these results,
in certain cases we are able to determine roots of unity which relate
almost characters to the characteristic functions. In particular we show how to compute
the values of all unipotent characters at all unipotent classes for the exceptional
groups of type $G_2$, $F_4$, $E_6$, $\lexp 2E_6$, $E_7$ and $E_8$ by a method different from
that of \cite{CV,K2}; we therefore require weaker restrictions on $p$ and
$q$. We also provide an appendix which
gives a complete list of the cuspidal character sheaves on all quasi-simple groups.
\end{abstract}

\maketitle

\section{Introduction} Let $\bG$ be a connected reductive algebraic group over a field
of positive characteristic $p$. We shall generally assume that $p$ is ``sufficiently large'',
which will often mean ``larger than the Coxeter number of the associated Weyl group''.
Let $F$ be a Frobenius morphism defining a rational structure on $\bG$ over the finite
extension $\Fq$ of the finite field $\BF_p$ with $p$ elements. We shall be be concerned with
the irreducible characters of the finite group of fixed points $\GF$ over the field $\Qlbar$,
where $\ell$ is a prime different from $p$.
The purpose of this work is to contribute to the determination of
the values of the irreducible characters of $\GF$ at unipotent elements of $\GF$.
In addition, we include, as an appendix (\S\ref{s:classcs})
a classification of the cuspidal character sheaves for quasi simple groups $\bG$, which
is complete up to a small number of ambiguities. This classification is essentially
due to Lusztig \cite{CS}, but we provide a list, conveniently arranged,
of cuspidal character sheaves for each isogeny type of quasi-simple group.
Specifically, we give a discussion of the
series of classical type in Appendix \S\ref{s:classcl},
and tables for the exceptional groups in Appendix \S\ref{s:classex}.

The irreducible characters of $\GF$ are partitioned into subsets in various ways.
The cuspidal character sheaves on Levi subgroups of $\bG$
 lead to a classification into ``Harish-Chandra type'' series,
whose constituents are labelled by (twisted) characters of an appropriate Coxeter group.
Each character of $\GF$ has a \wavefront\ and a unipotent support, both of which are
(geometric) unipotent conjugacy classes of $\bG$, and belongs to a `Lusztig series'.
The Lusztig series are further partitioned into families. Our results relate principally to
certain classes of characters which are described in terms of these partitions.

The notion of determination of a value requires explanation.
Lusztig has shown that the space $\CC(\GF)$ of class functions on $\GF$ has an
orthonormal basis consisting of the
characteristic functions of $F$-invariant simple $\bG$-equivariant perverse
sheaves on $\bG$. Such functions will be referred to as \lf s, and we consider their
values known by the work of Lusztig \cite{CS}.
Further, it was shown in \cite{Shoji1,Shoji2,Shoji3,Shoji4,B,W} that with certain qualifications, the \lf s
coincide with `almost characters' (see \S\ref{s:fourier}) up to multiplication by a
root of unity. Since the transition from almost characters to irreducible
characters is known, it follows that determination of this root of unity implies the determination
of the values of certain characters.

Our main results are as follows.

In  \S\ref{s:cf-ggg} we complement \cite{DLM3}  by giving an expression for
the  characteristic function  of an  arbitrary unipotent  class in  terms of
duals  of generalised Gelfand-Graev characters, and as a consequence deduce
some  results concerning the  support of these  duals. The special cases of
the regular and subregular classes are spelled out explicitly. We also give
a  formula  for  the  value  of  any irreducible character at any element of its unipotent
support.
In \S\ref{s:wavefront},
we provide some background on character sheaves and families, and define Lusztig series
and the \wavefront\ in a way suitable for our purpose. In \S\ref{ss:princbl} we determine
the multiplicity of an irreducible character $\chi$ of $\GF$ in the
Mellin transform $\Gamma_\ci$ of the \ggg s,
when $\ci$ is in the principal series, in terms of the Lusztig series of
$\chi$. This is applied to special cases such as subregular $\chi$, where
explicit formulae may be given.

Section \ref{s:resnil-cs} provides a general study of the restriction
to the unipotent set of
\lf s, which are defined as characteristic functions of $F$-stable character sheaves,
and applications to the determination of various
multiplicities.
The restriction to $\GFuni$ of the \lf \;$\chi_{E,\phi_E}$ (see \S\ref{ss:cf})
is given (Theorem \ref{prop:reslf} and Lemma \ref{lem:wf-cs})
in terms of Weyl group data and Green functions.
Item (ii) of Theorem \ref{prop:reslf} has also been obtained by Jay Taylor
in \cite{Taylor2}.

This is applied in Theorem
\ref{thm:gamma-cs} to give the inner product of a \lf\ with a \ggg.
These results are applied to prove vanishing theorems for inner products,
and in Corollary \ref{sch:gggch} to give an explicit expression for the projection of
$\Gamma_\ci$ to the space spanned by the characters with given \wavefront\
in terms of Weyl group data. In \S\ref{s:ls} this is applied to show that
when the \wavefront\ above is the support of $\ci$, then the projection
to the unipotent characters of the above projection is precisely a \lf, up to
sign. This explains phenomena which had been observed in several examples earlier.
We also show how the inner product of any irreducible character with $\Gamma_\ci$
may be expressed in terms of its inner products with \lf s.

Section \ref{s:fourier} treats a special situation which applies in the cases
when $\bG$ is of type $G_2,F_4$  or $E_8$.
In this situation, we are able to determine the root of unity which relates the almost characters
to the \lf s, by an analysis which uses the Mellin transforms (see Definition \ref{def:mellin})
of the characters in a family. Specifically we compute ({\cf} Theorem \ref{prop:projgamma-ac})
the projection of a \ggg\
to the space space spanned by a family with the same wave front set. In the last paragraph
of \S\ref{ss:setup} it is explained how this permits the computation of character values.

Finally, as mentioned above, in Appendices \ref{s:classcs},\ref{s:classcl} and
\ref{s:classex}, we present a complete classification of the cuspidal character
sheaves on all quasi-simple groups $\bG$.

\section{Characteristic functions and generalised Gelfand-Graev Characters.}\label{s:cf-ggg}

We maintain the notation of \cite{DLM3}, which we now briefly review.
Consider pairs $\ci=(C,\zeta)$, where $C$ is a
unipotent   class  of  $\bG$   and $\zeta$ is a  $\bG$-equivariant  irreducible
$\Qlbar$-local  system on $C$; $C$ is said to be the {\em support} of the pair and
may also be written $C_\ci$.
Set $a_\ci=|A(u)| $ for $u\in C$, where $A(u)=C_\bG(u)/C_\bG^0(u)$. Each
such  pair belongs to a cuspidal pair on  a Levi subgroup $\bL$ of $\bG$, and
all  pairs belonging to a  given cuspidal system  form a {\em  block} $\CI$; we
also denote $\bL$  by $\bL_\CI$.  When $\CI$ is $F$-stable we may choose
the cuspidal data to be $F$-stable and we let $\eta_\bL=(-1)^{\text{semisimple
$\Fq$-rank  of $\bL$}}$.

Suppose $X$ is an algebraic variety over $\Fqbar$, with an $\Fq$-structure embodied in
the Frobenius morphism $F:X\to X$. If $\chi$ is a complex of $\Qlbar$-sheaves on $X$
and we are given an isomorphism $\phi:F^*\chi\overset{\sim}{\lr}\chi$, then this data
permits the definition (see \cite[\S 2]{DLM3}) of a function $c_\phi:X^F\to \Qlbar$,
called the {\it characteristic function} of the pair $\chi,\phi$, where as usual, $X^F$ denotes
the set of $F$-stable points of $X$.
For an $F$-stable pair $\ci$, we denote  by $\CY_\ci$ the
characteristic function of
$\zeta$  and by $\CX_\ci$  the  characteristic  function of the corresponding
intersection  cohomology complex. Writing $c_\ci=\frac 12(\codim C-\dim
Z_\bL)$,  we define the normalisations $\wt\CY_\ci=q^{c_\ci}\CY_\ci$
and   $\wt\CX_\ci=q^{c_\ci}\CX_\ci$ of $\CY_\ci$ and $\CI_\ci$ respectively.
Define a  partial  order on pairs by stipulating
$\ci\le\ck$ if  the  pairs  are  in  the  same  block  and
$C_\ci\subset\overline C_\ck$. We have $\CX_\ci=\sum_{\ck\le\ci}
P_{\ck,\ci}\CY_\ck$    for    some    $P_{\ck,\ci}\in\BZ[q]$;   we   define
$P_{\ck,\ci}=0$  when $\ck\not\le\ci$,  and the  normalised version $\tilde
P_{\ck,\ci}=q^{c_\ci-c_\ck}P_{\ck,\ci}$ so that
$\wt\CX_\ci=\sum_{\ck\le\ci} \wt P_{\ck,\ci}\wt\CY_\ck$.

The fixed point set $C^F$ splits into $\GF$-classes indexed by the set
$H^1(F,A(u))$ of $F$-classes in $A(u)$, where $u$ is any (chosen) element of $C^F$.
For   $a\in  A(u)$   we  denote   by  $u_a$  a
representative  of the  $\GF$-class defined  by the  $F$-class of  $a$, and
denote  by $\Gamma_{u_a}$ the  Generalised Gelfand-Graev character attached
to the $\GF$-class of $u_a$. We let $\Gamma_\ci=\sum_{a\in
A(u)}\CY_\ci(u_a)\Gamma_{u_a}$ and define the normalisation
$\wt\Gamma_\ci=a_\ci\inv\zeta_\CI\Gamma_\ci$,
where $\zeta_\CI$ is the fourth root of unity attached by
Lusztig  to a  block $\CI$  (see \cite[7.2]{US}).
Finally  we  denote  by  $D$ (or $D_{\bG}$ when appropriate)
  the  Alvis-Curtis  duality  operation on $\CC(\GF)$ and by
$f\mapsto  f^*$  the  operation  on  Laurent  polynomials  in $q$ such that
$f^*(q)=f(q\inv)$. Note that we refer to the $\Gamma_\ci$ as the `Mellin transforms'
of the $\Gamma_{u_a}$; this should not be confused with the `Mellin transforms'
defined in Definition \ref{def:mellin} below.

For $\CI$ an $F$-stable block the sets
$(\CY_\ci)_{\ci\in\CI^F}$, $(\CX_\ci)_{\ci\in\CI^F}$ and
$(\Gamma_\ci)_{\ci\in\CI^F}$   are   three   bases   of  the  same  space
$\CC_\CI(\GF)$  of  unipotently  supported  class  functions. This space is
stable  under Alvis-Curtis duality  and two such  spaces attached to different
blocks are orthogonal.

Let  $\chi_{(x)}$ be the normalised characteristic function of the class in
$\GF$ of an element $x$, \ie, the function whose value is zero outside this
class and $|C_\bG(x)^F|$ on the class. The following result is a variation
on \cite[Lemma 2.5, Cor 2.6]{G}, which we shall require below.
\begin{theorem} \label{chi_u}
Let $u$ be a rational unipotent element. Then the normalised
characteristic function of the $\GF$-class of $u$ is given by
\begin{equation}\label{eq:charfn}
\chi_{(u)}=\sum_\CI\eta_{\bL_\CI}\zeta_\CI
\sum_{\ci\in\CI^F}\overline{\wt\CY_\ci(u)}
\sum_{\overset{\cj\geq\ci}{\cj\in\CI^F}}a_\cj\inv
(\wt P(\wt P^*)\inv)_{\ci,\cj}D\Gamma_\cj
\end{equation}
where $\CI$ runs over the rational blocks and $\wt P$ is the matrix whose
$(\ci,\cj)$ entry is the polynomial $\wt P_{\ci,\cj}$.
\end{theorem}
Note that the sum in the theorem is over those blocks which contain a local
system whose support is the class of $u$.
\begin{proof}
We  shall find  coefficients $m_{\ci,\cj}$  such that  the set of functions
$(f_\ci=\sum_\cj   m_{\ci,\cj}D\Gamma_\cj)$  is  the  basis  of  the  space
$\CC_\uni(\GF)$  of unipotently  supported class  functions dual to the
basis  $(\wt\CY_\ci)$ for the  usual inner product  $\scal f{f'}\GF$ on
class functions. It will then follow that
$\chi_{(u)}=\sum_\ci\overline{\wt\CY_\ci(u)}f_\ci=
\sum_\ci\overline{\wt\CY_\ci(u)}\sum_\cj m_{\ci,\cj}D\Gamma_\cj$.

The coefficients $m_{\ci,\cj}$ are determined by the equations $\sum_\cj
m_{\ci,\cj}\scal{D\Gamma_\cj}{\wt\CY_\cl}\GF=\delta_{\ci,\cl}$.       By
orthogonality of the spaces $\CC_\CI(\GF)$ for different blocks, we have
$m_{\ci,\cj}=0$ unless  $\ci$ and $\cj$ are in the  same block.
For any  total order  extending $\le$  the matrix $\tilde
P_{\ci,\cj}$ is upper unitriangular, thus invertible; using
$\wt\CY_\cl=\sum_{\ck\in\CI}\tilde    P'_{\ck,\cl}\wt\CX_\ck$   where
$\tilde  P'_{\ck,\cl}$ are the entries of the matrix $\wt P\inv$, we get
$\scal{D\Gamma_\cj}{\wt\CY_\cl}\GF =\sum_\ck\tilde
P'_{\ck,\cl}\scal{D\Gamma_\cj}{\wt\CX_\ck}\GF$ if $\cl$ and $\cj$ are in
the  same  block  $\CI$,  and  the  inner  product  is  0  otherwise. So by
\cite[proof  of  6.2]{DLM3}  we  get $\scal{D\Gamma_\cj}{\wt\CY_\cl}\GF=
\sum_\ck\tilde           P'_{\ck,\cl}\eta_{\bL_\CI}a_\cj\zeta_\CI\inv\tilde
P^*_{\cj,\ck}$  and  the  equations  for  the $m_{\ci,\cj}$, when $\ci$ and
$\cj$ are in a block $\CI$, are $\sum_\cj m_{\ci,\cj}
\eta_{\bL_\CI}a_\cj\zeta_\CI\inv      \sum_\ck\tilde     P'_{\ck,\cl}\tilde
P^*_{\cj,\ck}=\delta_{\ci,\cl}$.   As   $\sum_\ck\tilde  P'_{\ck,\cl}\tilde
P^*_{\cj,\ck}$  is the $(\cj,\cl)$ entry  of $\wt P^*\wt P\inv$, this
can  be  written  in  matrix  terms  as  follows, $M$ being the matrix with
entries $m_{\ci,\cj}$:
$$
M\begin{pmatrix}\ddots&&\cr&a_\cj&\cr&&\ddots \end{pmatrix}
\wt P^*\wt P\inv=\eta_{\bL_\CI}\zeta_\CI I,
$$
where $I$ is the unit matrix.
Hence $$m_{\ci,\cj}=a_\cj\inv\eta_{\bL_\CI}\zeta_\CI
(\wt P(\wt P^*)\inv)_{\ci,\cj},$$ as stated.
\end{proof}

\begin{scholium}\label{sch:supdg}
For any rational unipotent element $u$ the (virtual) character
$D_G\Gamma_u$ is supported by unipotent classes greater than
or equal to the class of $u$.
\end{scholium}
\begin{proof} Let $C$ be the geometric class of $u$.
The statement is equivalent to the analogous support property for the functions
$D_G(\Gamma_\ci)$ for all $\ci$ supported by $C$.
Using the triangular shape of the
matrices $(\wt P'_{\ck,\cl})$ and $(\wt P^*_{\ck,\cl})$
it follows
from the proof of Theorem \ref{chi_u}, that
$\scal{D\Gamma_\cj}{\wt\CY_\cl}\GF$ is zero unless $\cj\leq\cl$.
Since the characteristic function of the $G^F$-class of a unipotent
element $u$ is a linear combination of $\wt\CY_\cl$ with  $\cl$
running over pairs with support $C$,  the
result follows.
\end{proof}

The last statement is also a consequence of the following result.
\begin{corollary}
With notation as in Theorem \ref{chi_u}, we have
\be
\chi_{(u)}=
\sum_{\{C\subseteq\bG_\uni\mid u\in\ol C\}}
|A(C)|\inv\sum_{a\in A(C)}c(u,v_a)D_\bG\Gamma_{v_a},
\ee
where $C$ runs over the unipotent classes, $A(C)=A(v)$ for $v\in C$ and
$v_a\in C^F$ corresponds to $a \in A(C)$, and
$$
c(u,v_a)=\sum_{\CI\in\CP^F}\eta_{\bL_\CI}\zeta_\CI\inv\wt\CY_\CI(u)^t
\wt P(\wt P^*)\inv\wt\CY_\CI(v_a),
$$
and $\wt\CY_\CI(u)$ denotes the column vector with entries $\wt\CY_\ci(u)$
where $\ci$ runs over $\CI$.
\end{corollary}
\begin{proof}
We substitute the relation $\Gamma_\cj=\sum_{a\in A(v)}\CY_\cj(v_a)\Gamma_{v_a}$
into equation \eqref{eq:charfn} and rearrange.
\end{proof}

Scholium \ref{sch:supdg} follows from the above statement by simply inverting
the equation for $\chi_{(u)}$.
Now we specialise Theorem \ref{chi_u} to the case of a regular unipotent class;
let $\Gamma_u^\CI$ be the orthogonal
projection of $\Gamma_u$ onto the space $\CC_\CI(\GF)$
and write $c_\CI(u)$ for the
common value of $c_\ci$ for $\ci\in\CI$ whose support contains $u$:
\begin{corollary}\label{cor:regu}
If $u$ is a rational regular unipotent element, then
$$\chi_{(u)}=\sum_\CI\eta_{\bL_\CI}\zeta_\CI q^{c_\CI(u)}D\Gamma_u^\CI,$$
where the sum is over the regular blocks (those blocks containing a pair
whose support is the regular unipotent class).
\end{corollary}
\begin{proof}
Since the matrix $\wt P$ is unitriangular and the diagonal blocks
attached  to local systems  with a given  support are identity submatrices,
formula \ref{eq:charfn} reduces to:
$$\chi_{(u)}=\sum_\CI\eta_{\bL_\CI}\zeta_\CI
\sum_{\{\ci\in\CI^F\mid C_\ci\owns u\}}\overline{\wt\CY_\ci(u)}
a_\ci\inv D\Gamma_\ci.$$
Now substitute the value $\Gamma_u^\CI=
\sum_{\{\ci\in\CI^F\mid C_\ci\owns u\}}\overline\CY_\ci(u)
a_\ci\inv\Gamma_\ci$ given in \cite[lemma 6.3]{DLM3} to obtain the
result.
\end{proof}
The corresponding result for subregular elements is
\begin{corollary}\label{ratsubreg}
If $u$ is a rational subregular unipotent element
$$\chi_{(u)}=\sum_\CI\eta_{\bL_\CI}\zeta_\CI    q^{c_\CI(u)}[D\Gamma_u^\CI+
(q\inv-q)\sum_{\overset{\{\ci\in\CI\mid  C_\ci\owns  u \text{ and}}{ \text{ $\ci$  is
standard}\}}}   \overline{\CY_\ci(u)})\CY_{\ci_1}(v)D\Gamma_v^\CI],$$  where
$v$ is any regular rational unipotent element, $\ci_1$ is the pair in
$\CI$ labelled via the Springer correspondence by the trivial representation, and where
``standard'' has the sense of \cite[proposition 7.1]{DLM3}.
\end{corollary}
The sum above may be restricted to the blocks which contain a local system
supported by the subregular class.
\begin{proof}
Let  $\ci$ be a pair  with subregular support; \cite[proposition 7.1]{DLM3}
states  (once a misprint $q$ for $q\inv$ is corrected in (i)) that if $\ci$ is
standard  then $\wt P_{\ci,\cj}=\begin{cases}  q\inv &\text{if }\cj=\ci_1\\
\delta_{\ci,\cj}&\text{otherwise}\\   \end{cases}$   and   otherwise   $\wt
P_{\ci,\cj}=\delta_{\ci,\cj}$.  In the first case we can arrange the matrix
$\wt  P$ so  that it  is upper  unitriangular and  the lower  right corner,
indexed   by   $\ci$   and   $\ci_1$,  is  $\begin{pmatrix}  1&q\inv\\0&1\\
\end{pmatrix}$.  In the second  case we can  arrange $\wt P$  so that it is
upper  unitriangular and $\ci$ indexes the last line. The lower right block
of   $\wt   P(\wt   P^*)\inv$ in the two respective cases
is   then
$\begin{pmatrix}
1&q\inv-q\\0&1\\ \end{pmatrix}$ or $\begin{pmatrix} 1\\
\end{pmatrix}$.

Using these values, formula \ref{eq:charfn} reduces to:
$$\chi_{(u)}=\sum_\CI\eta_{\bL_\CI}\zeta_\CI \big (
\sum_{\{\ci\in\CI^F\mid C_\ci\owns u\}}
\dfrac{\overline{\wt\CY_\ci(u)}}{a_\ci} D\Gamma_\ci
+
\sum_{\overset{\{\ci\in\CI\mid  C_\ci\owns  u \text{ and}}{ \text{ $\ci$  is
standard}\}}}
\dfrac{\overline{\wt\CY_\ci(u)}}{a_{\ci_1}} (q\inv -q)D\Gamma_{\ci_1}
\big ).$$
The first term in the sum can be transformed as in Corollary \ref{cor:regu}.
If we take into account that there is at most one regularly supported
local system in a block (see \cite[1.10]{DLM2}), which in this case we take to be $\ci_1$,
then $D\Gamma_v^\CI=\overline\CY_{\ci_1}(v)a_{\ci_1}\inv D\Gamma_{\ci_1}$
which yields $a_{\ci_1}\inv D\Gamma_{\ci_1}=\CY_{\ci_1}(v)D\Gamma_v^\CI$
since $\CY_{\ci_1}(v)$ is a root of unity because $A(v)$ is commutative.
The second term is now as in the statement of the corollary .
\end{proof}

We  now look at the leading term  in the formula \eqref{eq:charfn}, \ie, the
term indexed by $\gamma$ such that $u$ is in the support of $\gamma$.
\begin{theorem}\label{DGammaiv}(cf. \cite[\S 2.4]{G})
If  $v$ is a rational unipotent element and $\ci$ a rational pair such that
$C_\ci\owns v$, we have
$$D\Gamma_\ci(v)=|A(v)||(C_\bG(v)^\circ)^F|\eta_{\bL_\CI}\zeta_\CI\inv
q^{-c_\ci}\CY_\ci(v),$$ where $\CI$ denotes the block of $\ci$.
\end{theorem}
\begin{proof}
Assume that $v$ is a rational unipotent element in the same geometric class
as  $u$. Since in  \ref{eq:charfn} the only  terms which do  not vanish are
those where $C_\ci\owns u$, thus $C_\ci\owns v$, and since
$D\Gamma_\cj(v)=0$  when $\cj>\ci$ and since the diagonal blocks in $\wt P$
pertaining to pairs with the same support are identity matrices, we get
\begin{equation}\label{chiuv}
\chi_{(u)}(v)=|C_\GF(u)|\delta_{(u),(v)}=
\sum_\CI\eta_{\bL_\CI}\zeta_\CI\sum_{\ci\in\CI^F}
\overline{\wt\CY_\ci(u)}a_\ci\inv D\Gamma_\ci(v).
\end{equation}
We now use the  orthogonality relation \cite[(4.2)]{DLM3} for the $\CY_\ci$
which can be written
$$\sum_{(u)}|A(u)^F|\inv\overline{\CY_\ci(u)}\CY_\cj(u)=\delta_{\ci.\cj},$$
where  $(u)$ runs over the rational  classes contained in $C_\ci$.
Multiplying both sides of the rightmost equation in \eqref{chiuv}
by $|A(u)^F|\inv\CY_\cj(u)$ and summing over $(u)$
we    get    $$|A(v)^F|\inv\CY_\cj(v)|C_\GF(v)|=    \eta_{\bL_\CI}\zeta_\CI
q^{c_\cj}a_\cj\inv  D\Gamma_\cj(v),$$ whenever
$C_\cj\owns v$, whence the theorem.
\end{proof}
The following corollary, which may be found in
\cite[2.4(a)]{G},
is also a direct consequence of the formulas in the
proof of Theorem \ref{chi_u}.
\begin{corollary} Let $\ci $ and $\cj$ be two pairs with same support; then
$$\scal{D\Gamma_\ci}{\CY_\cj}\GF=
\begin{cases} 0&\text{if }\ci\neq\cj\\ a_\ci\eta_{\bL_\CI}\zeta_\CI\inv
q^{-c_\ci} &\text{if } \ci=\cj\in\CI
\end{cases}$$
\end{corollary}
\begin{proof}
By Theorem \ref{DGammaiv} we have
$$\displaylines{
\scal{D\Gamma_\ci}{\CY_\cj}\GF=|\GF|\inv\sum_{v\in (\supp\ci)^F}D\Gamma_\ci(v)\overline{\CY_\cj(v)}=\cr
|\GF|\inv\sum_{v\in (\supp\ci)^F}|A(v)||A(v)^F|\inv|C_\GF(v)|\eta_{\bL_\CI}\zeta_\CI\inv
q^{-c_\ci}\CY_\ci(v)\overline{\CY_\cj(v)},}$$
where $\CI$ is the block containing $\ci$. The last sum reduces to
$$
\sum_{a\in A(v)}\eta_{\bL_\CI}\zeta_\CI\inv
q^{-c_\ci}\CY_\ci(u_a)\overline{\CY_\cj(u_a)}
$$
where   $u_a$  is  a   representative  of  the   rational  conjugacy  class
in $C_\ci$ parameterised by $a$, given the choice of $v\in C_\ci$.
Applying the orthogonality formula
\cite[(4.2)]{DLM3}, we obtain the result.
\end{proof}
\begin{remark}
If we apply the above theorem to a regular unipotent element we recover formula \cite[2.1]{DLM2}.
\end{remark}
Let us now compute the value of an irreducible character on its ``unipotent
support''  (see  Definition \ref{def:wfs}).  The  following  proposition generalises
\cite[3.15.4]{DLM1}.
\begin{proposition}\label{prop:val-us}
Let  $\chi$ be an irreducible character and $v$ a rational unipotent element
such  that  $\scal\chi{D\Gamma_u}\GF=0$  for  any  $u$ in a conjugacy class
strictly larger than the conjugacy class of $v$, then
$$\chi(v)=\sum_\CI\eta_{\bL_\CI}\zeta_\CI\inv q^{c_\CI(v)}\scal\chi{D\Gamma_v^\CI}\GF.$$
\end{proposition}
\begin{proof}
We have $\chi(v)=\scal\chi{\chi_{(v)}}\GF$. By Theorem \ref{chi_u} this is equal to
$$\sum_\CI\eta_{\bL_\CI}\zeta\inv_\CI\sum_{\ci\in\CI^F}\tilde
\CY_\ci(v)\sum_{\cj\geq\ci}a_\cj\inv(\wt P(\tilde
P^*)\inv)_{\ci,\cj}\scal\chi{D\Gamma_\cj}\GF.$$
In this formula, if $\ci\in\CI^F$ yields a non-zero summand, then $C_\ci\owns v$
and as $\cj\geq\ci$, the inner product in the sum is zero unless $\cj=\ci$.
Hence we have $\chi(v)=
\sum_\CI\eta_{\bL_\CI}\zeta\inv_\CI\sum_{\ci\in\CI^F}\tilde \CY_\ci(v)a_\ci\inv\scal\chi{D\Gamma_\ci}\GF$,
which can be written
$$
\sum_\CI\eta_{\bL_\CI}\zeta\inv_\CI q^{c_\CI(v)}\scal\chi{a_\ci\inv\sum_{\ci\in\CI^F}\overline{\CY_\ci(v)}
D\Gamma_\ci}\GF,
$$
which completes the proof.
\end{proof}

\section{Character sheaves, \wavefront, Lusztig series and Families}\label{s:wavefront}

We begin with some background.

\subsection{Character sheaves}\label{ss:cs} These arise as follows. Our notation is similar
to that in \cite{US} and \cite{AA}, with some significant departures.
Let $\bL$ be a Levi subgroup of $\bG$ and let $\ci_{\bar \bL}:=(C,\xi)$ be
a cuspidal local system in $\bL/Z^0(\bL)$, where $C$ is a conjugacy class of the
latter group and $\xi$ is a local system on $C$.
Let $\bar{\CS}$ be a Kummer local system on the torus $\bL/[\bL,\bL]$.
We may then form the local system $\ci_{\bar \bL}\boxtimes \bar{\CS}$
on $\bL/Z^0(\bL)\times\bL/[\bL,\bL]$. The pullback of this local
system under the map $\bL\to \bL/Z^0(\bL)\times\bL/[\bL,\bL]$
is supported on $\Sigma:=Z^0(\bL)C$, and we denote it by $\ci_\bL\otimes
\CS$, where $\ci_\bL$ and $\CS$ are respectively the pullbacks of
$\ci_{\bar\bL}$ and $\bar\CS$.
The intersection complex extension $\IC(\ci_\bL\otimes \CS)[\dim\Sigma]$ is then a
cuspidal character sheaf \cite[3.12]{CS} \cite[2.5]{ICC},
and it is known \cite{Lu12} that this intersection complex is clean,
and that therefore is supported on $\Sigma$, and is equal to
$\ci_\bL\otimes \CS[\dim\Sigma]$ there.

We may now form the induced character sheaf ({\cf} \cite[(8.1.2), p. 237]{CS}), referred to
as $K$ in {\it loc.~cit.}, but which we shall also write
$\Ind_\bL^\bG((\ci_\bL\otimes \CS)[\dim\Sigma])$.
Lusztig has shown \cite[3.4]{ICC} that $\End(\Ind_\bL^\bG((\ci_\bL\otimes \CS)[\dim\Sigma]))\simeq\CA$, a finite dimensional
$\Qlbar$-algebra, isomorphic to a twisted group algebra of $W_\bG(\bL,\ci_\bL,\CS)$, the subgroup
of the relative Weyl group $\WGL=N_\bG(\bL)/\bL$ which fixes the Kummer system $\CS$ on $\bL$, as well as
the cuspidal pair $\ci_\bL$ described above.

It follows that
\be\label{eq:dec-ind}
K=\Ind_\bL^\bG((\ci_\bL\otimes \CS)[\dim\Sigma])\simeq
\oplus_{E\in\Irr(\CA)}A_{\ci_\bL,\CS,E}\otimes V_E,
\ee
where $\Irr(\CA)$ denotes the set of irreducible characters of $\CA$,
and for each $E\in\Irr(\CA)$,
$A_{\ci_\bL,\CS,E}$ is an irreducible character sheaf on $\bG$ and
$V_E$ is a $\Qlbar$-representation of $\CA$, with character $E$.
Generally, the data $\bL,\ci_\bL$ and $\CS$ will be fixed, and when there is
no risk of confusion, we write $A_E$ for $A_{\ci_\bL,\CS,E}$.
\begin{remark}\label{rem:cs-cusp}
We shall denote character sheaves by $A_E$, but will sometimes need to refer to the associated
cuspidal data. In that case, we write $(\bL,\ci_\bL,\CS)(E)=(\bL(E),\ci_\bL(E),\CS(E))$
for the relevant data.
\end{remark}

\subsection{Characteristic functions}\label{ss:cf} Now suppose that $\bL$ and $\Sigma$
above are $F$-stable. The stabiliser $\{wF\in\WGL F\mid
F^*\dot w^*(\ci_\bL\otimes \CS)\simeq (\ci_\bL\otimes \CS)\}$ is a subcoset
$W_\bG(\bL,\CS)w_1F\leq \WGL F$. Now for $E$ such that there is an isomorphism
${F}^*A_E\overset{\overset{\phi}{\sim}}{\lr}A_E$, we have an associated characteristic
function $\chi_{E,\phi}:\bG^{F}\to\Qlbar$, defined as an alternating sum on
the cohomology of $A_E$ in the usual way.

In the above situation, there is an isomorphism $\phi_0:F^*
\dot w_1^*K\to K$, which permutes the
canonical decomposition \eqref{eq:dec-ind}. Note that $F^*\dot w_1^*$ acts
on this decomposition as $F^*\ot F^*\dot w_1^*$.
Thus for each $E\in\Irr(\CA)$ such that
$F^*\dot w_1^*E\simeq E$, $\phi_0$ restricts to
$\phi_E\otimes\sigma_E\in\End(A_E\otimes V_E)$, and hence for each
choice of $\sigma_E$ defines
$\phi_E:F^*A_E\to A_E$. The associated characteristic function
$\chi_{E,\phi_E}:\GF\to\Qlbar$ is what we refer to as a \lf. The
various \lf s, suitably normalised,
form an orthonormal basis of the space of class functions on
$\GF$. For further details, see \cite[10.4, 10.6]{CS}.

\subsection{Families, Lusztig series and the \wavefront}\label{ss:wff}
The set $\hat\bG$ of character sheaves on $\bG$ is partitioned into families:
$\hat \bG=\amalg_{\CL,c}\hat\bG_{\CL,c}$, where $\CL$ is a Kummer system on
a fixed maximal torus $\bT$ of $\bG$ and $c$ is a family  in the group
$W_\bG(\CL)$ (see \cite[16.7 and 17.4]{CS} for this partition and
the definition of two-sided cells and
families in this group), two such pairs being considered equivalent if they are
conjugate under the Weyl group. Now it is shown in \cite[Thm. 10.7]{US} that
\begin{proposition}\label{family support}
Given  a family $(\CL,c)$ there  is a unique unipotent  class $C$, called the {\em
unipotent support} of the family, such that for any character sheaf $A_E\in
\hat\bG_{\CL,c}$,  its stalk at $g=su\in\bG$ (Jordan decomposition) is zero
if $\dim(u)\geq\dim C$ and $u\notin C$, and there exists $u\in C$ such that
the  stalk at $u$ is nonzero for some $A_E\in \hat\bG_{\CL,c}$.
\end{proposition}

The last statement is a consequence of \cite[(g)  page  172]{US}.

The \lf s correspond to the $F$-stable sheaves in the $F$-stable
families $(\CL,c)$, and we may therefore identify the set of Lusztig
functions with $\hat\GF$ and consequently have a partition of this set as
$$
\cfs=\amalg_{(\CL,c):F(\CL,c)=(\CL,c)}\cfs_{\CL,c},
$$
and the spaces spanned by
the distinct $\cfs_{\CL,c}$ are orthogonal.

Correspondingly, there is a partition ({\cf} \cite[11.1]{US}) of the
irreducible characters, whose parts we call again families:
$$\Irr\GF= \amalg_{(\CL,c):F(\CL,c)=(\CL,c)}\Irr\GF_{(\CL,c)}.$$
These  partitions, both of characters and of character sheaves, are
defined  by the blocks of the matrix $\scal{\rho}{\chi_{E,\phi_E}}\GF$.
This allows us to identify families of irreducible characters and families of
character sheaves.
Thus  specifically,  $\scal{\rho}{\chi_{E,\phi_E}}\GF\ne  0$  only  if
both $\rho$ and $A_E$ belong to the family parameterised by the pair $(\CL,c)$.

\begin{definition}\label{def:ls}
The  {\em  (Lusztig)  series}  of  an  irreducible  character $\rho$ (resp.\
character sheaf $A_E$) is said to be the Kummer local system $\CL$ on $\bT$
if $\rho\in\Irr(\GF)_{\CL,c}$ (resp.\ $A_E\in \hat \bG_{\CL,c}$).

We say that $\rho$ (resp.\ $A_E$ or $\chi_{E,\phi_E}$) is {\em unipotent} if its
series  $\CL$ is equal to the trivial sheaf $\Qlbar$.
\end{definition}

Note  that the  Kummer system  $\CL$ on  the maximal  torus $\bT$  of $\bG$
corresponds  to a semisimple element $s\in\bG^*$,  the group dual to $\bG$.
We  may therefore write  $(s,c)$   for  the   family  $(\CL,c)$,  and  for
$s\in{\bG^*}^F$ we denote by $\CE(\GF,s)$ the Lusztig series
$\cup_c\Irr(\GF)_{s,c}$ of irreducible characters.
The   unipotent   characters   correspond  to  $s=1$,  or
equivalently $\CL=\Qlbar$.

Let $\chi$  be  an  irreducible  character  of  $\GF$. Lusztig has shown
\cite[11.2]{US} that the following definitions make sense.

\begin{definition}\label{def:wfs}
\begin{enumerate}
\item
The {\em \wavefront\ $\wf(\chi)$} of $\chi$ is the  largest unipotent class $C$
such that $\chi$ is  a component of the corresponding generalised Gelfand-Graev
representation $\Gamma_u$ for some $u\in C^F$.
\item The {\em unipotent support} of $\chi$ is the largest unipotent class $C$
such that $\chi$ has a non zero value on some element with unipotent part in
$C^F$.
\end{enumerate}
\end {definition}

In  the above, ``largest'' means that for any class of higher dimension, or
of  same dimension as  $C$ but different  from $C$, the multiplicity (resp.\
the  value) is $0$.
\begin{remark}\label{rem:us}
Proposition \ref{prop:val-us} applies in particular to any irreducible character
$\chi$ whose unipotent support is the class $(v)$. For in that case, if $(u)>(v)$, then
by Scholium \ref{sch:supdg}
$D\Gamma_u$ is supported on unipotent classes $(u')$ with $(u')\geq(u)>(v)$,
and $\chi$ vanishes on such classes, whence  $\scal\chi{D\Gamma_u}\GF=0$.
\end{remark}

Given a family $c$ of $W_\bG(\CL)$, let $c\ot\ve$ denote the family defined
by  the property that  $\psi\in c$ if  and only if $\psi\ot\ve\in c\ot\ve$,
where $\ve$ is the alternating character of $W$ restricted to $W_\bG(\CL)$.
With this notation, Lusztig \cite{US} has proved the following properties
of the sets defined above.
\begin{proposition}\label{usupp}
Let $\chi\in\Irr(\GF)_{(\CL,c)}$ be an irreducible character and let $C$ be
its unipotent support and $C'$ its \wavefront.
Then
\begin{enumerate}
\item There exists $u\in C^F$ such that $\chi(u)\ne 0$.
\item $C$ is the \wavefront\ of the Alvis-Curtis dual of $\chi$.
\item $C'$ is the unipotent support of $(\CL, c\ot\ve)$.
\end{enumerate}
\end{proposition}
\begin{proof}
(i) is \cite[Thm. 11.2(v)]{US},
(ii) is \cite[Thm. 11.2(iv)]{US}.
(iii) is \cite[Thm. 11.2(i) and (iii)]{US}, taking into account
\cite[11.1]{US}.
\end{proof}

One would expect that in the  situation of Proposition \ref{usupp}, $C$ is
the  unipotent support  of $(\CL,c)$; that is, any character $\chi\in\Irr(\GF)_{(\CL,c)}$
has the same unipotent support, and this support is the unipotent support of
$(\CL,c)$.  We shall  not require  this in  the current work.
It would be a consequence of the assertion that if $\chi\in\Irr(\GF)_{(\CL,c)}$
then the Alvis-Curtis dual $D\chi$ of $\chi$ is in
$\Irr(\GF)_{(\CL,c\ot\ve)}$.  If  $\bG$  has  connected  center this last fact
follows from \cite[6.14]{livre}.

%
%
%
%

In view of the properties \ref{usupp}, we make the following definition.
\begin{definition}\label{wf(E)}
The  {\em \wavefront}  of a  family $(\CL,c)$  is the unipotent support of
$(\CL,c\otimes\ve)$;  equivalently,  it is  the  common \wavefront\ of all
characters  in  $\Irr(\GF)_{\CL,c}$.  If  $A_E\in\hat\bG_{\CL,c}$  we shall
denote by $\wf(E)$ the \wavefront\ of its family and $\supp(E)$ the
unipotent support of
its family.
\end{definition}

The following explicit description of the map between families
and their \wavefront\ may be found in \cite[10.5, 10.6]{US}.
Let $(\CL,c)$ be a family, and let $s\in\bG^*$ be a semi-simple element
corresponding to $\CL$. Let us
write $W'(s)\simeq W_\bG(\CL)$ for the Weyl group of the
not necessarily  connected  group  $C_{\bG^*}(s)$. The  Weyl  group of the
identity  component is denoted by $W(s)$. Let $E$
be  a special representation of $c\ot\ve$; by definition, its restriction
to  $W(s)$ is a sum  of special representations; let  $E_1$ be one of these
and consider $j_{W(s)}^W(E_1)$. This is an irreducible representation of
$W$, and its Springer correspondent is supported by a unipotent class which is
independent of the choice of $E_1$ and is the unipotent support of
$(\CL,c\ot\ve)$ \ie\ the \wavefront\ of $(\CL,c)$.

Conversely, given a unipotent class $C$, we obtain
Lusztig  families having $C$ as \wavefront\  as  follows.
First note, that for a sufficiently
large power $F^m$ of $F$, since the \ggg s
of $G^{F^m}$ form a basis of the space of unipotently supported class
functions, each unipotent class is the \wavefront\
of  some irreducible character, hence of some family $(\CL,c)$.
Accordingly, if $E'$  is the Springer correspondent  of the  pair $(C,\Qlbar)$,
then using the above description of the \wavefront, we see that
there is always an $s$ such that the  $j$-restriction of
$E'$ to $W(s)$ is a  non zero special  representation $E$ of $W(s)$. For
any  such $s$ this $E$  defines a family of  $W(s)$, and hence families $c'$ of
$W'(s)$.  The class $C$ is now the \wavefront\ of any  character in one of the
families $(s,c'\otimes \ve)$.

\begin{definition}
We say that a character is regular (resp.\ subregular) if its \wavefront\ is the
regular (resp.\ a subregular) class.
\end{definition}

Given the above description of the \wavefront, we may characterise the
subregular characters in a Lusztig series as follows.

\begin{lemma}\label{lem:subregchar} Let $\bG$ be simply connected.
For  any $s$ the  subregular characters in  the Lusztig series $\CE(\GF,s)$
are precisely the characters in the family $(s, c')$, where $c'$ is a family of
$W'(s)$ whose restriction to $W(s)$
contains one of the characters $\ve\otimes r_i$ where $r_i$ is the
reflection representation of the $i$-th irreducible component of $W(s)$.
\end{lemma}
\begin{proof}
A  case-by-case check (see  for example table  \cite[4.1]{DLM3}) shows that
when  $W$ is irreducible, the trivial  local system on the unique subregular
class $C$ corresponds to the reflection representation, a special character
of  $W$ (for this last  fact the reflection representation  is alone in its
family  in simply laced types by \cite{RR}; for the other types one may use
the description in \cite{SR}). It follows that in general the components of
the  reflection  representation  correspond  to  the  trivial system on the
various subregular classes.

But the restriction to $W(s)$ of the
reflection representation is the sum of the $r_i$ and
$(\rank(W)-\rank(W(s)))$  times the identity; hence the $j$-restriction
of the reflection representation is
the sum of the $r_i$, and the result follows.
\end{proof}

\subsection{Multiplicities in generalised Gelfand-Graev representations}\label{ss:princbl}

Let  $\CI$ be the  principal block, attached to the cuspidal pair
$(\bT,\ci_0)$, where $\ci_0$ is the pair $(1,\Qlbar)$.
The  aim of this subsection is to compute some
multiplicities  of characters in  $\Gamma_\ci$  for  $\ci  \in\CI$.

For $s\in(\bG^*)^F$, denote  by $w_1F$ the type of a maximally split torus of the centraliser of
$s$.  For  a  $W(s)$-class  function  $f$  on  $W(s)w_1F$ we define a class
function  in  $\CE(\GF,s)$  by  $R_f:=\frac  1{|W(s)|}\sum_{y\in  W(s)w_1F}
f(y)R_{\bT_y}^\bG(\theta)$,    where    $R_{\bT_y}^\bG(\theta)$    is   the
Deligne-Lusztig  character  where  for  $y\in  WF$  we denote by $\bT_y$ an
$F$-stable  torus such that $(\bT_y,F)$ is $\bG$-conjugate to $(\bT,y)$ and
$\theta\in\Irr(\bT_y^F)$ corresponds to $s\in{\bG^*}^F$.

\begin{proposition}\label{prop:projgamma}
For $\ci$ in the principal block $\CI$, the projection of $\Gamma_\ci$ onto
$\CE(\GF,s)$ is, with notation as above
\begin{equation}\label{eq:projgamma}
a_\ci\sum_{\cj\in\CI^F}\wt P^*_{\ci,\cj}
R_{\Res^{WF}_{W(s)w_1F}(\wt{\vp_\cj}\otimes\wt\ve)},
\end{equation}
where  $\vp_\cj\in\Irr(W)^F$  is  the  Springer  correspondent of $\cj$ and
$\wt{\vp_\cj}$ and $\tilde\ve$ are preferred extensions  to $WF$, in particular
$\tilde\ve(vF)=\ve(v)$ (see \cite[17.2]{CS}).
\end{proposition}
\begin{proof}
Lusztig's formulas
\cite[\S7.5  (b) and p.176  (b), proof of  11.2]{US}, suitably modified for
the  non-split case may be applied to give  \cite[Proposition 6.1]{DLM3} which, applied to the
principal block yields
$$\Gamma_\ci=a_\ci Q^\bG(\tilde\ve\CZ_\bT\wt{Q_\ci^*}),$$
where $\CZ_\bT$ is
the function on $WF$ given by
$\CZ_\bT(y)=y\mapsto|\bT_y^F|$,
$\wt{Q_\ci}(y)=\sum_{\cj\in\CI^F}\wt{\vp_\cj}(y)\wt P_{\ci,\cj}$,
and $Q^\bG$ is the map which sends the function $f$ to
$\frac 1{|W|}\sum_{y\in WF}f(y)R^\bG_{\bT_y}(\CX_{(1,\Qlbar),w})$. But
$\CX_{(1,\Qlbar),w}=|\bT_y ^F|\inv\sum_{\theta\in\Irr(\bT_y^F)}\theta$
is the characteristic
function of the identity (see \cite[Definition 3.1(iii) and
Proposition 3.2]{DLM3}). It follows that
$$
\Gamma_\ci=\frac{a_\ci}{|W|}\sum_{\cj\in\CI^F}\wt P^*_{\ci,\cj}
\sum_{y\in WF,\theta\in\Irr(\bT_y^F)}(\wt{\vp_\cj}\otimes\tilde\ve)(y)
R_{\bT_y}^\bG(\theta).
$$
Identifying characters with elements of the dual group, the inner sum can be
written
$$
\sum_{y\in WF,t\in\bT^{*y}}(\wt{\vp_\cj}\otimes\tilde\ve)(y)
R_{\bT_y}^\bG(t).
$$
Let  $\chi$ be an irreducible character in the series $\CE(\GF,s)$. Then in
the expansion of the inner  product  $\scal{\Gamma_\ci}\chi\GF$,  only  the summands where
$(y,t)$  is $W$-conjugate to some $(y_0,s)$ could be non  zero, and
if $(y,t)= (y_0,s)$ we have $y_0\in W(s)w_1F$. Further, a term
$(\wt{\vp_\cj}\otimes\tilde\ve)(y)R_{\bT_y}^\bG(t)$ depends only on the
$W$-conjugacy class of $(y,t)$.
Since the number of pairs $(y,t)$
conjugate to $(y_0,s)$ is $\frac{|W|}{|C_{W(s)}(y_0)|}$ and the number of
pairs $(y',s)$ conjugate to $(y_0,s)$ is
$\frac{|W(s)|}{|C_{W(s)}(y_0)|}$, we get
$$
\scal{\Gamma_\ci}\chi\GF=
\frac{a_\ci}{|W(s)|}\sum_{\cj\in\CI^F,y'\in
W(s)w_1F}(\wt\vp_\cj\otimes\wt\ve)(y)\wt P^*_{\ci,\cj}
\scal{R_{\bT_{y'}}^\bG(s)}\chi\GF,
$$
whence the proposition follows, given the definition of $R_f$.
\end{proof}

\begin{proposition}\label{prop:wf-ac-pb} Let $\chi\in\CE(\GF,s)$ have
\wavefront\ $C$, and let $\ci$ be in the principal block. Then
$$\scal{R_{\Res^{WF}_{W(s)w_1F}(\wt{\vp_\ci}\otimes\wt\ve)}}\chi\GF=0$$
unless $\dim\supp\ci<\dim C$ or $\supp\ci=C$.
\end{proposition}
\begin{proof}
Assume $\scal{R_{\Res^{WF}_{W(s)w_1F}(\wt{\vp_\ci}\otimes\wt\ve)}}\chi\GF\ne0$
and let $(s,c)$ be the family of $\chi$, which is also (see
section  \ref{ss:wff}) the family of some component $\psi$ of
$\Res^{WF}_{W(s)w_1F}(\wt{\vp_\ci}\otimes\wt\ve)$.
For a unipotent class $C$,
let us denote by $\beta_C$ the dimension of the variety of
Borel subgroups containing an  element of $C$.
By \cite[Corollary 10.9 (i)]{US} applied with $E'=\psi\otimes\ve$, we have
$\beta_{\supp\ci}\geq a(c\otimes\ve)$. We also know by part (h) of
the proof of \cite[Thm. 10.7]{US} that $a(c\otimes\ve)=\beta_C$ where
$C$ is the class $C(s,c\otimes\ve)$ as in \cite[10.5]{US}, that is (see the
beginning of section \ref{s:wavefront}) $C$ is the \wavefront\ of
all characters in the family $c$. So
if $\scal{R_{\Res^{WF}_{W(s)w_1F}(\wt{\vp_\ci}\otimes\wt\ve)}}\chi\GF\neq 0$
then  $\dim\supp\ci\leq \dim C$.  Moreover by \cite[Corollary
10.9(ii)]{US},
if equality pertains, then $\supp\ci=C$.
\end{proof}
Using the above propositions we now prove
\begin{corollary} \label{ci subreg}
The assumptions and notation being as in the above proposition,
let $\ci$ be a pair in the principal block.
\begin{enumerate}
\item If $\chi\in\CE(\GF,s)$ has \wavefront\ $\supp\ci$ then
$$\scal{\Gamma_\ci}\chi\GF=
a_\ci\scal{R_{\Res^{WF}_{W(s)w_1F}(\wt{\vp_\ci}\otimes\wt\ve)}}\chi\GF$$
\item
If $\ci$ has subregular support, $\chi\in\CE(\GF,s)$ is regular, and
$\bG$ has connected centre, then
$$\scal{\Gamma_\ci}\chi\GF= \begin{cases}
a_\ci(q+\scal{\Res^{WF}_{W(s)w_1F}\wt{\vp_\ci}}\Id{W(s)w_1F})&
\text{if $\ci$ is standard}\\
a_\ci\scal{\Res^{WF}_{W(s)w_1F}\wt{\vp_\ci}}\Id{W(s)w_1F}&
\text{otherwise.}
\end{cases}
$$
\end{enumerate}
\end{corollary}
When  $\chi$  is  subregular,  in the  notation  of
Lemma \ref{lem:subregchar}, we can write (i) above as
$$a_\ci\sum_E\scal{\Res^{WF}_{W(s)w_1F}(\wt\vp_\ci\otimes\wt\ve)}{\tilde
E}{W(s)w_1F}\scal{R_{\tilde   E}}\chi\GF,$$  where   $E$  runs   over  the
$w_1F$-invariant characters of $W(s)$ which are in the family $c'$.

\begin{proof}
We  first  prove  (i).  If  $\supp\ci$  is  the  \wavefront\  of  $\chi$,  by
Proposition \ref{prop:wf-ac-pb} we know that the $\cj$ in
\eqref{eq:projgamma}  giving rise to  a nonzero scalar  product with $\chi$
must   have   support   smaller   than   $\supp\ci$.   On  the  other  hand
$P^*_{\ci,\cj}$  is $0$ unless the support of $\cj$ is greater than that of
$\ci$; then since the diagonal blocks  of $P^*_{\ci,\cj}$ are the identity,
only the term in the formula survives.

We  now prove  (ii). If  $\ci$ has  subregular support  then (see  proof of
Corollary \ref{ratsubreg})  apart from $\wt P^*_{\ci,\ci}=1$  the only other non-zero
$\wt  P^*_{\ci,\cj}$ occurs when $\ci$ is standard and $\cj$ is the trivial
local  system on the regular class,  in which case $\vp_{\cj}=\Id$ and $\wt
P^*_{\ci,\cj}=q$.  Moreover, since  the centre  of $\bG$  is connected, the
regular character of the series $s$ is equal to $R_{\tilde\ve}$. This gives
the result.
\end{proof}
In case  the  centre  of  $\bG$  is connected, any regular character is
orthogonal to $\CC_\CI(\GF)$ for all non principal $\CI$
since the only local system on the regular
class is the trivial one. This leads to the following result.

\begin{proposition}Assume that the centre of $\bG$ is connected. Let $s$ be
a  semisimple element  of $(\bG^*)^F$  and let  $\chi_s$ denote the regular
character  of  $\CE(\GF,s)$.  Then  for  any  subregular rational unipotent
element $u$ we have
$$\scal{\Gamma_u}{\chi_s}\GF=
q\sum_\cj \overline{\CY_\cj(u)}+
\sum_\ci\overline{\CY_\ci(u)}
\scal{\Res^{WF}_{W(s)w_1F}\wt\vp_\ci}\Id{W(s)w_1F},$$
where the first sum is over the standard pairs with subregular support in
the  principal  block  and  the  second  sum  is  over  all pairs in the
principal block with subregular support.
\end{proposition}
\begin{proof}
By the orthogonality relations for the $\CY_\ci$ (see for example
\cite[4.2]{DLM3}) we have
\begin{equation}\label{gammau}
\Gamma_u=|A(u)|\inv\sum_\ci\overline{\CY_\ci(u)}\Gamma_\ci,
\end{equation}
whence the result, by Corollary \ref{ci subreg}(ii) and the above remark.
\end{proof}
In  the particular case where there is only one pair in the principal block
with subregular support, the above formula becomes
$$\scal{\Gamma_u}{\chi_s}\GF=q+\scal{\Res^{WF}_{W(s)w_1F}\wt{r}}\Id{W(s)w_1F},$$
where $r$ is the reflection representation of $W$.

\section{Restriction of character sheaves to the unipotent set.}\label{s:resnil-cs}

In \cite{DLM3} the interrelationships among various bases of the space
$\CC_\uni(\GF)$ of unipotently supported functions is discussed, and
Lusztig induction and restriction were described in these terms. In this
section we discuss the restriction of the characteristic function
of an arbitrary Frobenius-stable character sheaf to
the unipotent set of $\GF$, and in particular give a formula for its
inner product with the generalised Gelfand-Graev character $\Gamma_u$.
Such characteristic functions are referred to as \lf s (see \S\ref{ss:cf}).

We  begin by describing  the restriction of  the \lf s $\chi_{E,\phi_E}$ to
the  unipotent set $\GFuni$.  It is known  by \cite[pp. 151]{CV} that
when  the local system $\ci_{\bar\bL}$  as in \S\ref{ss:cs} has unipotent
support,  then  $W_\bG(\bL,\ci_\bL)  =W_\bG(\bL)$,  and  that  the  implied
cocycle is trivial, so that $\CA\simeq\Qlbar W_\bG(\bL,\CS)$ and the \lf s in the
corresponding block can be indexed by characters of $W_\bG(\bL,\CS)$.

Item (ii) of the following theorem
has also been obtained by J.~Taylor (\cite[Theorem 7.7]{Taylor2})
by completely different methods.

\begin{theorem}\label{prop:reslf}
Let $\chi_{E,\phi_E}$ be a \lf\ with associated cuspidal
data $(\bL,\ci_\bL,\CS)$ as in \S\ref{ss:cs} and \S\ref{ss:cf}. Then
\begin{enumerate}
\item  If $\ci_{\bar\bL}$ has support which  is not a unipotent class, then
the restriction $\chi_{E,\phi_E}|_\GFuni$ is zero.
\item  If $\ci_{\bar\bL}$  has unipotent  support, then  $E\in\Irr \CA=\Irr
W_\bG(\bL,\CS)$. In this case, we have
\be\label{resuni}
{\chi_{E,\phi_E}}|_\GFuni=(-1)^{\dim\supp\ci_\bL}Q^\bG
\left(\Ind_{W_\bG(\bL,\CS).w_1F}^{W_\bG(\bL).F}(\ol{\tilde E})\right),
\ee
where $\tilde E$ is a suitable extension of $E$ and $Q^\bG$ is the map from
$\CC(\WGL.F)$ to $\CC_{\CI(E)}(\GF)$ defined by
$Q^\bG(\wt\vp_\ck)=\wt\CX_\ck$  (see \cite[Def.  3.1]{DLM3}). Here $\CI(E)$
is  the  block  corresponding  to  the unipotently supported cuspidal local
system $\ci_{\bar\bL}$.
\end{enumerate}
\end{theorem}
\begin{proof}
The first part is in \cite[Th. 8.5]{CS}.
We therefore assume that $\ci_{\bar\bL}$ has unipotent support,
so that $\CA=\Qlbar W_\bG(\bL,\CS)$. Then for any element $w\in W_\bG(\bL,\CS)$,
we have an isomorphism $\phi_{0,w}^*:F^*(ww_1)^*K\to K$, and, for each
$E$ such that $F^*\dot w_1^*E\simeq E$, an associated
characteristic function $\chi_{E,\phi_E,w}$. Moreover those $E\in\Irr W_\bG(\bL,\CS)$
which are fixed by $F^*\dot w_1^*$ are precisely the same as those fixed by
$F^*(\dot w\dot w_1)^*$. Following Lusztig \cite[\S10.4]{CS}, for the characteristic
functions, for each $w\in W_\bG(\bL,\CS)$, we have
$$
\chi_{K,\phi_{0,w}}=
\sum_{E:\dot w_1^*F^*E= E}\tilde E(ww_1F)\chi_{E,\phi_E},
$$
where $\tilde E$ is a suitable extension of $E$ to $W_\bG(\bL,\CS).w_1F$.
Inverting this relation using the orthogonality of characters of cosets,
we obtain as in \cite[formula (10.4.5)]{CS}
$$
\chi_{E,\phi_E}=|W_\bG(\bL,\CS)|\inv
\sum_{w\in W_\bG(\bL,\CS)}\ol{\tilde E(ww_1F)}\chi_{K,\phi_{0,w}}.
$$
But again by \cite[Th. 8.5]{CS} (or \cite[Prop. 3.2]{DLM3}) the restriction
of $\chi_{K,\phi_{0,w}}$ to the unipotent set is
$(-1)^{\dim\supp\ci_\bL}Q^\bG_{ww_1}$  (the sign  comes from  the fact that
the  perverse sheaf is shifted from  the intersection cohomology complex by
$\dim\supp\ci_\bL$),    and   by   \cite[Def.   3.1(iii)]{DLM3}   we   have
$Q^\bG_{ww_1}=Q^\bG(\gamma_{ww_1})$,    where   $\gamma_v$    denotes   the
normalised  characteristic function  of the  class of  $vF$ in $\WGL.F$. It
follows that
$$
\Res^\GF_\GFuni(\chi_{E,\phi_E})=(-1)^{\dim\supp\ci_\bL}
|W_\bG(\bL,\CS)|\inv
\sum_{v\in W_\bG(\bL,\CS).ww_1F}\ol{\tilde E(v)}Q^\bG(\gamma_{ww_1}).
$$
But by Frobenius' formula for induced characters,
$$
|W_\bG(\bL,\CS)|\inv\sum_{v\in W_\bG(\bL,\CS).ww_1F}
\ol{\tilde E(v)}\gamma_{ww_1}=\Ind_{W_\bG(\bL,\CS).w_1F}^{\WGL.F}(\ol{\tilde{E}}),
$$
and the result follows.
\end{proof}

\begin{remark}\label{rem:cstoblock}  Theorem  \ref{prop:reslf}  shows  that
given  a \lf\ $\chi_{E,\phi_E}$  with non-zero restriction  to the unipotent
set,  its  restriction  lies  in  $\CC_{\CI(E)}$  for  a well defined block
$\CI(E)$,  which corresponds  to the  cuspidal local system $\ci_{\bar\bL}(E)$,
which has unipotent support.
\end{remark}

\subsection{Wave  front  set  and  character  sheaves}  In  order to better
understand the restriction to the unipotent set of a \lf, we shall need the
following  information concerning  the {\lf}s  which have non-trivial inner
product with a given irreducible character.

\begin{definition}\label{def:n_E} With notation as in Theorem \ref{prop:reslf},
we define $n_\ck(E)$ by
$$\Ind_{W_\bG(\bL,\CS).w_1F}^{\WGL.F}(\overline{\tilde E})=
(-1)^{\dim\supp\ci_\bL}\sum_{\ck\in\CI(E)}n_\ck(E) \tilde\vp_\ck.
$$
\end{definition}
Note that $n_\ck(E)$ depends on the extension $\tilde E$ which itself depends on
the chosen isomorphism $\phi_E$.

\begin{remark}\label{rem:nk}
\begin{enumerate}
\item  For any $F$-stable character sheaf $A_E$, the scalars $n_\ck(E)$ are
defined   by   \ref{def:n_E},   which   implicitly   involves   the  triple
$(\bL(E),\ci_\bL(E),\CS(E))$  as described in Remark \ref{rem:cs-cusp}, and
assumes  that $\ci_{\bar\bL}(E)$ has unipotent support, as is the case when
the  restriction of  $\chi_{E,\phi_E}$ to  the unipotent  set of $\GF$ is
non-zero  by  Theorem  \ref{prop:reslf}(i).  Whenever  we  use the notation
$n_\ck(E)$  we  shall  take  this  as  understood.  \item  The  $n_\ck$ are
algebraic integers. They will figure prominently in the rest of this work.
\end{enumerate}
\end{remark}

The following key result generalises Proposition \ref{prop:wf-ac-pb}, which
deals only with the principal block. Recall that in \cite[3.6]{DLM3} we defined
an involution $\ci\mapsto\hat\ci$ on a block $\CI$ and a sign $\ve_\ci$ on $\CI$
by $\wt\vp_\ci\otimes\wt\ve=\ve_\ci\wt\vp_{\hat\ci}$ where $\wt\ve$ is defined
on $W_\bG(\bL)F$ by  $\wt\ve(wF)=\ve(w)$. We have $\ve_\ci=1$ when
$\bG$ is split (that is, $F$ acts trivially on $W$).
\begin{lemma}\label{lem:wf-cs}
Let  $A_E$  be  a  character  sheaf  in  $\cfs_{\CL,c}$,  and  suppose that
$\chi_{E,\phi_E}$ has non-zero restriction to the unipotent set. Let $C$ be
the \wavefront\ of $(\CL,c)$ and $C'$ its unipotent support. If $n_\ck(E)\neq 0$, then
\begin{enumerate}
\item either the support of
$\ck$ is $C'$, or it has dimension smaller than $\dim(C')$.
\item either the support of
$\hat\ck$ is $C$, or it has dimension smaller than $\dim(C)$.
\end{enumerate}
\end{lemma}
\begin{proof}
By Theorem \ref{prop:reslf} and Definition \ref{def:n_E} we have
\be\label{eq:uni-res-cs}
\Res^\GF_\GFuni\chi_{E,\phi_E}=\sum_{\ck\in\CI(E)}
n_\ck(E)\wt\CX_\ck.
\ee
The assertion (i) is now immediate from the definition of the support, since the transition matrix between the bases $\CX_\ci$ and $\CY_\ci$ is unitriangular.

We now apply the duality functor $D$ to \eqref{eq:uni-res-cs}, bearing in mind that duality commutes with restriction to the unipotent set and \cite[3.14 (ii)]{DLM3}, obtaining
\begin{equation}\label{res_dual}
\Res^\GF_\GFuni D(\chi_{E,\phi_E})=\sum_{\ck\in\CI(E)}
n_\ck(E)\eta_{\bL(E)}\ve_\ck\wt\CX_{\hat\ck}.
\end{equation}

Now $\chi_{E,\phi_E}$ is a linear combination of characters in the same family:
$$\chi_{E,\phi_E}=\sum_{\rho\in\Irr(\GF)_{\CL,c}}m_\rho\rho,$$
whence
$$D(\chi_{E,\phi_E})=\sum_{\rho\in\Irr(\GF)_{\CL,c}}m_\rho D(\rho),$$
where all the $D(\rho)$ have unipotent support $C$, since all the $\rho$ have \wavefront\
$C$. Again using the fact that the transition matrix between the bases $\CX_\ci$ and
$\CY_\ci$ is triangular, this means that the restriction to the unipotent set
of $D(\chi_{E,\phi_E})$ is a linear combination of $\CX_\ci$ with
the support of $\ci$ either equal to $C$ or having smaller dimension. Comparing
with \ref{res_dual}, we obtain (ii).
\end{proof}



\subsection{Inner product with Generalised Gelfand-Graev characters}\label{ss:cs-ggg}
We  begin with  the following  general expression  for the Mellin transform
$\Gamma_\ci$  in  terms  of  \lf  s.  Note  that  since  the  \lf s form an
orthonormal  basis of $\CC(\GF)$, it suffices for this purpose to compute
the  inner product of $\Gamma_\ci$ with an arbitrary \lf\ $\chi_{E,\phi_E}$.
Moreover  such an inner product is evidently zero unless the restriction of
$\chi_{E,\phi_E}$  to  $\GFuni$  is  non-zero,  in  which  case $A_E$
determines a unipotently supported cuspidal local system
$\ci_{\bar\bL}(E)$,  a corresponding  block $\CI(E)$  and a unipotent class
$\wf(E)$, viz. its \wavefront\ (see Definition \ref{wf(E)}).

The following result is related to Proposition \ref{prop:projgamma}.

\begin{theorem}\label{thm:gamma-cs}
Let  $\ci\in\CI^F$, a block with cuspidal datum $(\bL,\ci_\bL)$ and suppose
$A_E\in   \hat   G^F$.   In terms of the  integers   $n_\ck(E)$   of  Definition
\ref{def:n_E}, we have
\be\label{eq:gammai-cs}
\scal{\wt\Gamma_\ci}{\chi_{E,\phi_E}}\GF=
\begin{cases}
0 \text{ if }\Res^\GF_\GFuni(\chi_{E,\phi_E})=0
\text{ or }\CI(E)\neq\CI\\
\sum_{\overset{\ck\in\CI^F,\;\ck\geq\ci\text{ and}}{\overset{\dim(\supp(\ck))
<\dim(\wf(E))}{\text{ or }\supp(\ck)=\wf(E)}}}n_{\hat\ck}(E)\ve_\ck
\wt P_{\ci\ck}^*\text { otherwise.}\\
\end{cases}
\ee
Similarly
\be\label{eq:Dgammai-cs}
\scal{D(\wt\Gamma_\ci)}{\chi_{E,\phi_E}}\GF=
\begin{cases}
0 \text{ if }\Res^\GF_\GFuni(\chi_{E,\phi_E})=0
\text{ or }\CI(E)\neq\CI\\
\eta_\bL\sum_{\overset{\ck\in\CI^F,\;\ck\geq\ci\text{ and}}
{\overset{\dim(\supp(\ck))<\dim(\supp(E))}{\text{ or }\supp(\ck)=\supp(E)}}}
n_\ck(E)\wt P_{\ci\ck}^*\text { otherwise.}\\
\end{cases}
\ee
\end{theorem}
\begin{proof}
Using equation \eqref{eq:uni-res-cs} we get
$$
\scal{\wt\Gamma_\ci}{\chi_{E,\phi_E}}\GF
=\scal{\wt\Gamma_\ci}{\Res^\GF_\GFuni(\chi_{E,\phi_E})}\GF\\
=\sum_{\ck\in\CI(E)}n_\ck(E)\scal{\wt\Gamma_\ci}{\wt\CX_\ck}\GF.
$$
In the proof of \cite[6.2]{DLM3} one finds the equation
$\eta_\bL\ve_\ck\scal{\Gamma_\ci}{\tilde\CX_{\hat\ck}}\GF=
\eta_\bL a_\ci\zeta_\CI\inv\tilde P^*_{\ci,\ck}$
which by the definition of $\wt\Gamma_\ci$ can be written
$\scal{\wt\Gamma_\ci}{\tilde\CX_{\hat\ck}}\GF=\ve_\ck\tilde P^*_{\ci,\ck}$.

Changing  the variable  from $\ck$  to $\hat\ck$  in the equation above and
substituting,  we obtain \ref{eq:gammai-cs}:  the condition on $\supp(\ck)$
comes  from Lemma  \ref{lem:wf-cs}(ii) and  the condition $\ck\geq\ci$ from
the fact that $P_{\ci,\cj}$ is upper triangular.

For formula \ref{eq:Dgammai-cs}, since duality commutes with restriction to
the  unipotent  set,  the  function  $D(\wt\Gamma_\ci)$ is also unipotently
supported  so we  have the  same equation  with $\wt\Gamma_\ci$ replaced by
$D(\wt\Gamma_\ci)$.  We  then  use  \cite[3.14(ii)]{DLM3} which states that
$D(\tilde\CX_\ck)=\eta_\bL\ve_\ck\tilde\CX_{\hat\ck}$  and  proceed  in the
same way, using \ref{lem:wf-cs}(i) instead of \ref{lem:wf-cs}(ii).
\end{proof}

\begin{corollary}\label{cor:wfsupp}
With notation as in Theorem \ref{thm:gamma-cs}, if
$\dim(\supp(\ci))\geq\dim(\wf(E))$ then
\be\label{eq:wfsupp}
\scal{\wt\Gamma_\ci}{\chi_{E,\phi_E}}\GF=
\begin{cases}
n_{\hat\ci}(E)\ve_\ci\text{ if }\supp(\ci)= \wf(E)\text{ and }\CI(E)=\CI,\\
0 \text { otherwise.}\\
\end{cases}
\ee
and if $\dim(\supp(\ci))\geq\dim(\supp(E))$ then
\be\label{eq:Dwfsupp}
\scal{D(\wt\Gamma_\ci)}{\chi_{E,\phi_E}}\GF=
\begin{cases}
\eta_\bL n_\ci(E)\text{ if }\supp(\ci)= \supp(E)\text{ and }\CI(E)=\CI,\\
0 \text { otherwise.}\\
\end{cases}
\ee
\end{corollary}
\begin{proof}
The condition on the first case of equation \ref{eq:wfsupp}
(resp.\ equation \ref{eq:Dwfsupp})
is such that  there is only one summand
satisfying the conditions in the sum on the right side of
\eqref{eq:gammai-cs}  (resp.\ of \eqref{eq:Dgammai-cs}),  and it corresponds
to  $\ck=\ci$.  In  the  second  case,  there  is no $\ck$ satisfying those
conditions.
\end{proof}

We next prove the ``Mellin transform'' of  the statement in Theorem \ref{thm:gamma-cs}.

\begin{corollary}\label{thm:cs-ggg} Let $A_E$ be a character sheaf such that
$\Res^\GF_\GFuni\chi_{E,\phi_E}\neq 0$. Let $C$ be a unipotent
class in $\bG$ and $u\in C^F$.
Then, again in terms of the integers $n_\ck(E)$ of
Definition \ref{def:n_E}, we have
\be\label{eq:cs-ggg}
\scal{\Gamma_u}{\chi_{E,\phi_E}}\GF=
\zeta_{\CI(E)}\inv\sum_{\{\ci\in\CI(E)\mid\supp(\ci)=C\}}\ol{\CY_\ci}(u)
\sum_{\overset{\ck\in\CI(E)^F,\;\ck\geq\ci\text{ and}}
{\overset{\dim(\supp(\ck))<\dim(\wf(E))}{\text{ or }\supp(\ck)=\wf(E)}}}
n_{\hat\ck}(E)\ve_\ck\wt P_{\ci\ck}^*.
\ee
\end{corollary}
\begin{proof}
Using \ref{gammau} and the definition
$\Gamma_\ci=|A(u)|\zeta_\CI\inv\wt\Gamma_\ci$, we have
$$
\Gamma_u=\sum_{\ci\in\CP^F}\zeta_{\CI(\ci)}\inv\ol{\CY_\ci(u)}\wt\Gamma_\ci,
$$
where $\CI(\ci)$ is the block to which $\ci$ belongs.
The corollary now follows immediately from Theorem \ref{thm:gamma-cs}.
\end{proof}


The next corollary gives the multiplicity of a \lf\ in a \ggg\ corresponding to
a unipotent element of its \wavefront.

\begin{corollary}\label{cor:cs-ggg-wf} Maintain the notation of Corollary
\ref{thm:cs-ggg} and assume that $u\in \wf(E)^F$. Then
$$
\scal{\Gamma_u}{\chi_{E,\phi_E}}\GF=
\zeta_\CI\inv\sum_{\{\ci\in\CI(E)\mid\supp(\ci)=\wf(E)\}}n_{\hat\ci}(E)
\ve_{\ci}\ol{\CY}_{\ci}(u),
$$
\end{corollary}
\begin{proof}
We  use  the  relation  \eqref{eq:cs-ggg}  to  compute the left side. Since
$u\in\wf(E)$,  it follows that in the outer  sum all the $\ci$ have support
$\wf(E)$,  and hence in  the inner  sum of  \eqref{eq:cs-ggg}, there is just one
summand, viz. the term corresponding to $\ck=\ci$. The result follows.
\end{proof}

The next result is a consequence of Corollary \ref{cor:wfsupp}.
\begin{corollary}\label{sch:gggch}   Given  a  unipotent   class  $C$,  let
$\CC_{\wf(C)}(\GF)$  (resp.\  $\CC_{\supp(C)}(\GF)$)  be  the space of class
functions   on   $\GF$   which   has   basis   the  irreducible  characters
$\rho\in\Irr(\GF)_{(\CL,c)}$  where $C$ is the wavefront of $(\CL,c)$ (resp
the support of $(\CL,c)$). Write $\Proj_{\wf(C)}$ (resp.
$\Proj_{\supp(C)}$)  for the projection onto this space with respect to the
complement   spanned  by   the  other   irreducible  characters.   If
$\ck\in\CI^F$ (a block with cuspidal datum $(\bL,\ci_\bL)$) has support $C$,
then we have
$$
\Proj_{\wf(C)}\wt\Gamma_\ck=\ve_\ck\sum_{
{\overset{\{A_E\in(\hat G)^F\mid \CI(E)=\CI}{\text{ and }\wf(E)=C\}}}}
n_{\hat \ck}(E)\chi_{E,\phi_E}.
$$
and
$$
\Proj_{\supp(C)}D(\wt\Gamma_\ck)=\eta_\bL\sum_{
{\overset{\{A_E\in(\hat G)^F\mid \CI(E)=\CI}{\text{ and }\supp(E)=C\}}}}
n_\ck(E)\chi_{E,\phi_E}.
$$
\end{corollary}

\begin{proof}
Observe first that since the \lf s form an orthonormal basis of the
space of class functions on $\GF$, we have
$$
\Proj_{\wf(C)}\wt\Gamma_\ck=\sum_{\{E\mid\wf(E)=C\}}
\scal{\wt\Gamma_\ck}{\chi_{E,\phi_E}}\GF\chi_{E,\phi_E}.
$$
Now, since we have $\supp(\ck)=\wf(E)$ for the $E$ in the sum,
we can apply \ref{eq:wfsupp} and we get the stated value for
$\Proj_{\wf(C)}\wt\Gamma_\ck$.

The proof for $\Proj_{\supp(C)}D(\wt\Gamma_\ck)$ proceeds similarly, using
\ref{eq:Dwfsupp} instead of \ref{eq:wfsupp}.
\end{proof}

\section{Lusztig series and unipotent characters}\label{s:ls}

Recall (Definition \ref{def:ls}) the definition of the Lusztig series of an
irreducible character or character sheaf.
We shall require the following two statements concerning the
series to which a character sheaf belongs. This series is defined as above by a
Kummer local system $\CT$ on a maximal torus $\bT$ of $\bG$ \cite[2.10]{CS}.

\begin{lemma}\label{lem:series}
\begin{enumerate}
\item  Suppose  $(\bL,\ci_\bL,\CS)$  is  a  triple as in \S\ref{ss:cs}. Let
$\CT_\CS$  be the  lift to  $\bT$ of  $\CS$ and  let $\CT_0$ be the Lusztig
series  of  the  cuspidal  triple  $(\bL,\ci_\bL,\Qlbar)$. Then the Lusztig
series  of any  character sheaf  $A_E$ induced  from $(\bL,\ci_\bL,\CS)$ is
$\CT_0\otimes\CT_\CS$.  \item Suppose  $A_E$ and  $A_{E'}$ are respectively
character sheaves induced from $(\bL,\ci_\bL,\CS)$ and
$(\bL,\ci_\bL,\CS')$. If $A_E$ and $A_{E'}$ are in the same Lusztig series,
then $\CS$ and $\CS'$ are conjugate under $\WGL$.
\end{enumerate}
\end{lemma}
\begin{proof}
The  statement  (i)  may  be  found  in  \cite[17.9]{CS}.  Given  this, the
hypothesis of (ii) implies that for some $w\in \WGL$,
$\CT_0\otimes\CT_{\CS'}  =(\CT_0\otimes\CT_\CS)^w$.  But  $\CT_0$ is stable
under  $\WGL$  since  $W_\bG(\bL,\ci_\bL)=\WGL$  (see  beginning of section
\ref{s:resnil-cs}). It follows that
$\CT_0\otimes\CT_{\CS'}=\CT_0\otimes(\CT_\CS)^w$.  As  $\CT_0$  is a Kummer
system, we may multiply by its inverse, obtaining the result.
\end{proof}

\begin{lemma}\label{lem:unip}
Maintain the notation of Lemma \ref{lem:series}; in particular, $\CT_0$ is
the Lusztig series of the cuspidal triple $(\bL,\ci_\bL,\Qlbar)$.
\begin{enumerate}
\item  If $\ci_{\bar\bL}$ is unipotently  supported, then either $\CT_0$ is
trivial,  or  else  it  corresponds  to  an  element  $t_0\in  \bL^*$ whose
centraliser is not the whole of $\bL^*$.
\item  If $\ci_{\bar\bL}$ is unipotently supported and $A_E$ is a unipotent
character  sheaf induced  from $(\bL,\ci_\bL,\CS)$,  then both  $\CT_0$ and
$\CS$ are trivial Kummer systems.
\end{enumerate}
\end{lemma}
\begin{proof}
The  first  statement  is  verified  by  a  case  by  case check, using the
classification   of  cuspidal  character  sheaves  given  in  the  Appendix
below. Note that the restriction that the characteristic
be good is necessary for the truth of the assertion (i).

To see (ii), note that by Lemma \ref{lem:series} (i), the Lusztig series of $A_E$ is a Kummer system
which corresponds to the element $t_0s\in\bT^*$, the torus dual to $\bT$, where $t_0$ is as given,
and $s$ corresponds to $\CS$. Since $\CS$ is trivial on $[\bL,\bL]$, $s$ is centralised by $\bL^*$.
Now $A_E$ is unipotent precisely when $t_0s=1$; but by (i), if $t_0\neq 1$, then $t_0$ is not centralised
by $\bL^*$, whence $t_0s\neq 1$. Thus if $A_E$ is unipotent, $t_0=s=1$.
\end{proof}

The restriction to the unipotent set is simpler for unipotent character
sheaves:

\begin{lemma} \label{resuniuni}
Let $A_E$ be a unipotent character sheaf induced from
$(\bL,\ci_\bL,\CS)$. If $\phi_E$ is such that $\tilde E$ of
\ref{prop:reslf}(ii) is the preferred extension, then we have
$$\chi_{E,\phi_E}|_\GFuni=
\begin{cases} 0,&\text{if $\ci_{\bar\bL}$ does not have unipotent support}\\
(-1)^{\dim\supp\ci_\bL}\tilde\CX_\ck&\text{otherwise, where
$\varphi_\ck=E$.}\\
\end{cases}$$
\end{lemma}
\begin{proof}
The fact that the restriction is $0$ if $\ci_{\bar\bL}$ does not have
unipotent support results from \ref{prop:reslf}(i).

If we now assume that $\ci_{\bar\bL}$ is unipotently supported then
by Lemma \ref{lem:unip}(ii) both $\CS$ and $\CT_0$ of Lemma
\ref{lem:series}(i) are trivial.

Further, by \cite[9.2]{ICC}, since  $\ci_{\bar\bL}$ is unipotently
supported, then $W_\bG(\bL,\ci_\bL)=W_\bG(\bL)$. Formula
\ref{resuni} reduces thus to the statement of the lemma, taking in account
that the preferred extension takes rational values.
\end{proof}
Denote  by  $\Proj_\uni$  the  projection  onto  the  space spanned by the
unipotent  characters of $\GF$ with respect  to the complement spanned by
the  other characters, that is, the  orthogonal projection onto $\CE(\GF,1)$.

The next
result  shows that a  certain projection of  $\wt\Gamma_\ci$ is precisely a
\lf.  We say that a  block $\CI$ is in  the unipotent Lusztig series if its
cuspidal  datum is $(\bL,\ci_\bL)$  and the Lusztig  series of a character
sheaf  with cuspidal  datum $(\bL,\ci_\bL,\Qlbar)$  is the unipotent series
($\CT_0=\Qlbar$ in the language of Lemma \ref{lem:series}).

\begin{proposition}\label{prop:projunigamma}
Let  $\ck\in\CI$  be  a  pair  with  support  the unipotent class $C$,
and assume $(\bL,\ci_\bL)$ is the cuspidal data of $\CI$. Then
$$\Proj_\uni\Proj_{\wf(C)}(\wt\Gamma_\ck)=\begin{cases}
(-1)^{\dim\supp\ci_\bL}\ve_\ck  \chi_{E,\phi_E}
&\parbox[t]{4.2cm}{if $\CI$ is in the unipotent Lusztig series,}
\\ 0& \text{otherwise}\end{cases}$$
where $A_E$   is  the  character  sheaf  with cuspidal  data
$(\bL,\ci_\bL,\Qlbar)$  determined by
$$E:=\vp_{\hat \ck}\in\Irr(W_\bG(\bL,\Qlbar))=\Irr(W_\bG(\bL)).$$
Similarly, we have
$$\Proj_\uni\Proj_{\supp(C)}(D(\wt\Gamma_\ck))=\begin{cases}
(-1)^{\dim\supp\ci_\bL}\eta_\bL  \chi_{E,\phi_E}\hskip -1.5cm &\\
&\parbox[t]{4.7cm}{if $\CI$ is in the unipotent Lusztig series,}
\\ 0&\text{otherwise}\end{cases}$$
where $A_E$   is  determined by $E:=\vp_\ck\in\Irr(W_\bG(\bL))$.
\end{proposition}
\begin{proof}
In the formula for $\Proj_{\wf(u)}\wt\Gamma_\ck$ in Corollary
\ref{sch:gggch}, for any $E$ in the right-hand side,
$\Proj_\uni(\chi_{E,\phi_E})$
is either $\chi_{E,\phi_E}$ or zero, according as $A_E$ is unipotent or not.
In the first case, by Lemma \ref{resuniuni} there is just one
$E$ for which $n_{\hat\ck}(E)\neq 0$, namely $E=\vp_{\hat\ck}$, and for this
$E$, we have $n_{\hat\ck}(E)= (-1)^{\dim\supp\ci_\bL}$.
This completes the computation of $\Proj_\uni\Proj_{\wf(C)}(\wt\Gamma_\ck)$.

The computation of
$\Proj_\uni\Proj_{\supp(C)}(D(\wt\Gamma_\ck))$ proceeds similarly.
\end{proof}
\begin{corollary}\label{D(chi_E)}
Let $A_E$ be a unipotent character  sheaf  with cuspidal  data
$(\bL,\ci_\bL,\Qlbar)$  determined by $E:=\vp_{\hat \ck}\in\Irr(W_\bG(\bL))$.
Then $D(\chi_{E,\phi_E})=\ve_\ck\eta_\bL\chi_{E',\phi_{E'}}$ where $A_{E'}$
is the unipotent character sheaf with same cuspidal datum determined
by $E':=\vp_\ck\in\Irr(W_\bG(\bL))$.
\end{corollary}
\begin{proof} The statement follows by comparing the dual of
the first expression given in \ref{prop:projunigamma} with the second
expression, given that the dual of a family
of unipotent characters  with wavefront $C$
is a family of unipotent characters with support $C$.
We prove this latter fact.

By Proposition \ref{usupp}(iii), if $C$ is the wavefront of a family
$(\CL, c)$, then it is the support of the family $(\CL, c\ot\ve)$.
Since  for $E\in \Irr(W)$ we have
$D(R_E)=R_{E\ot\ve}$ (notation as in the second paragraph of
subsection \ref{ss:princbl})
and a unipotent character is in a family $(\Qlbar,c)$ if and
only if it has non-zero multiplicity in $R_E$ for some $E$ in that family,
it follows that the dual of the family $(\Qlbar,c)$  is the family
$(\Qlbar,c\ot\ve$).
\end{proof}

Fix the datum $(\bL,\ci_\bL,\CS)$ as described in \S\ref{ss:cs}.
Given an element $\theta=\sum_{E\in\Irr(W_\bG(\bL,\CS))}c_E E$
of the character ring of $W_\bG(\bL,\CS)$ over $\Qlbar$, define
$$
\chi_\theta:=\sum_{E\in\Irr(W_\bG(\bL,\CS))}c_E{\chi_{E,\phi_E}}.
$$

We shall now generalise Corollary \ref{ci subreg} to the case of
characters not necessarily in the principal series.

\begin{proposition}\label{cor:ggg-char}
Let  $\rho\in\Irr(\GF)$ have \wavefront\ $C$, and  let $\ck$ be a pair with
support $C$ in the block $\CI$, which has cuspidal datum $(\bL,\ci_\bL)$.
Then
\begin{enumerate}
\item All \lf s $\chi_{E,\phi_E}$ satisfying $\scal\rho{\chi_{E,\phi_E}}\GF
\scal{\wt{\Gamma}_\ck}{\chi_{E,\phi_E}}\GF  \neq  0$,  arise  from a
single datum $(\bL,\ci_\bL,\CS)$ as in \S\ref{ss:cs}.
\item With $(\bL,\ci_\bL,\CS)$ as in (i), we have
$$
\scal{\wt{\Gamma}_\ck}{\rho}\GF=
\ve_\ck\scal{\rho}{\chi_{ \Res^{W_\bG(\bL).F}_{W_\bG(\bL,\CS).w_1F}
\tilde\vp_{\hat\ck}}}\GF.
$$
\end{enumerate}
\end{proposition}
\begin{proof}
We first prove (i). If $\scal\rho{\chi_{E,\phi_E}}\GF\ne 0$ then
$C=\wf(\rho)=\wf(E)$ and from Corollary \ref{sch:gggch}(i) if further
$\scal{\wt{\Gamma}_\ck}{\chi_{E,\phi_E}}\GF\neq 0$ then
$\CI(E)=\CI$. Thus the part $(\bL,\ci_\bL)$ of the cuspidal datum for
$A_E$ is determined by the block $\CI$. Furthermore, $\scal\rho{\chi_{E,\phi_E}}\GF\ne 0$
implies that the Lusztig series of $A_E$ is determined by, and equal to,
that of $\rho$, so that by Lemma \ref{lem:series} the Kummer system $\CS$
in the datum for $A_E$ is also determined by the stated conditions.

We now prove (ii). Assume that $\rho\in\Irr(\GF)_{\CT,c}$.
Then $\rho$ may be expressed as
$$
\rho=\sum_{\chi_{E,\phi_E}\in \cfs_{\CT,c}}\scal{\rho}{\chi_{E,\phi_E}}\GF\chi_{E,\phi_E}.
$$
It follows using Corollary \ref{sch:gggch} that
\be\label{eq:comp}
\begin{aligned}
\scal{\wt{\Gamma}_\ck}{\rho}\GF
&=\sum_{\chi_{E,\phi_E}\in \cfs_{\CT,c}}
\scal{\rho}{\chi_{E,\phi_E}}\GF\scal{\wt{\Gamma}_\ck}{\chi_{E,\phi_E}}\GF\\
&=\sum_{\chi_{E,\phi_E}\in \cfs_{\CT,c}}
\scal{\rho}{\chi_{E,\phi_E}}\GF n_{\hat\ck}(E)\ve_\ck,\\
\end{aligned}
\ee
where by (i) the sum is restricted to those $E$ with a given cuspidal datum.
Hence the above sum is over $E\in\Irr(W_\bG(\bL,\CS).w_1F)$, and by Frobenius reciprocity,
$n_{\hat\ck}(E)=$
$\scal{E}{\Res^{W_\bG(\bL).F}_{W_\bG(\bL,\CS).w_1F}\tilde\vp_{\hat\ck}}{W_\bG(\bL,\CS).w_1F}$.
Thus
\be\label{eq:comp2}
\scal{\wt{\Gamma}_\ck}{\rho}\GF
=\ve_\ck\sum_{\chi_{E,\phi_E}\in \cfs_{\CT,c}}
\scal{\rho}{\chi_{E,\phi_E}}\GF\scal{E}{\Res^{W_\bG(\bL).F}_{W_\bG(\bL,\CS).w_1F}
\tilde\vp_{\hat\ck}}{W_\bG(\bL,\CS).w_1F},
\ee
and the proof is complete.
\end{proof}

\section{Character values, projections, generalised Gelfand-Graev characters and families.}
\label{s:fourier}
\subsection{The setup}\label{ss:setup}
In  this  section  we  are  interested  in  the projection of Gelfand-Graev
characters   onto the space of  unipotent  characters  and   in  the  value  of  unipotent
characters; since unipotent characters factor through the adjoint group, we
assume $\bG$ adjoint in this section.

We  shall explore the  following particular situation  (cf.\ \cite{FS}). We
consider  a family $\CF=(\CL,c)$ of unipotent characters, so that $\CL=\Qlbar$
and  $c$ is a family  in $W$ (see Definition  \ref{def:ls} and beginning of
section  \ref{ss:wff}). The \wavefront\  of such a  family is called a {\em
special} unipotent class $(u)$. We suppose that

\be\label{assume}
\begin{aligned}
&\text{(i) The group $\CG$ attached to $\CF$ (see below) is $A(u)$.}\\
&\text{(ii) At most one of the local systems on $(u)$ is not in the principal
block.}\\
\end{aligned}
\ee
Note that
\begin{itemize}
\item
In  general  the  group  $\CG$  attached  to  $\CF$  is a certain canonical
quotient  $\overline{A(u)}$  defined  in  \cite[after 13.1.2]{livre}; see also
\cite{FS}. Thus condition (i) amounts to stipulating that
$A(u)=\overline{A(u)}$.
\item
Since  $\bG$ has connected centre,  it is always possible  to choose $u$ in
its geometric class such that $F$ acts trivially on $A(u)$, see \cite[Prop.
2.4]{Taylor}.  We shall assume this property for the rest of this section, and hence label the
rational  classes contained  in $(u)$  as $(u_g)$, where $g$
runs over a set of representatives of the conjugacy classes of $A(u)$.
\end{itemize}

Applying  the discussion after Definition \ref{wf(E)}  to the
particular case of the unipotent series, one
sees  that there is a unique family  with wavefront $(u)$. Indeed, $(u)$ is
defined  by  the  fact  that  the  Springer  correspondent  of  the special
character of the family is the local system $((u),\Qlbar)$.

\begin{remark}\label{lem:assume}
The  assumptions  \eqref{assume}  are  satisfied  by all special unipotent
classes in groups of type $G_2$, $F_4$ or $E_8$ such that
$A(u)=\overline{A(u)}$  (that is, most special  classes). In fact for these
groups  all local  systems are  in the  principal block  except one  on the
subregular  class in type $G_2$ ($A(u)=\Sym_3$), on the class $F_4(a_3)$ in
type  $F_4$ ($A(u)=\Sym_4$), and on the class $E_8(a_7)$ (in the Bala-Carter
notation) in type $E_8$ ($A(u)=\Sym_5$). Assumptions \eqref{assume} are also
satisfied in groups of type $E_6$ and $E_7$ where all local systems
are in the principal block.
\end{remark}

In   the  next  subsection,  we  shall  define  Mellin  transforms  of  the
irreducible  characters in  $\CF$, and  show that  in the  cases covered by
Remark  \ref{lem:assume}, the  orthogonal projection  of $\Gamma_{u_g}$ onto
the  family $\CF$ is one  of these Mellin transforms.  A consequence of the
proof  will be  the determination  of all  the values  of the characters in
$\CF$  at unipotent  classes in  the cases  of Remark \ref{lem:assume};  then, any
unipotent  irreducible character $\chi$ is either in $\CF$ or else it is in
the principal series in the sense of character sheaves. In the latter case,
the values of $\chi$ on the unipotent set are given by Green functions, and
may be determined algorithmically.

Thus  a consequence of the results of  this section is the determination of
all  the  values  of  all  unipotent  irreducible  characters of $\GF$ at
unipotent classes for $\bG$ a group
of type $G_2$, $F_4$, $E_6$, $\lexp 2E_6$, $E_7$ and $E_8$.

\subsection{Families, Fourier transform and almost characters}
We assume henceforth that $\bG$ is quasi-simple (or equivalently that $W$ is irreducible).
We  give  the  basic  facts  concerning  the Fourier transform of unipotent
characters,  following \cite[\S 3, Ch.  VII]{DM1} and \cite[4.4]{DM2}. With
each  family $\CF$  of unipotent  irreducible characters  of $\GF$,
there  is associated a finite group  $\CG$. The unipotent characters in the
family  $\CF$ are  parameterised by  $$\CM(\CG):=\{(x,\chi)\mid x\in\CG, \;
\chi\in\Irr(C_{\CG}(x))\}/\CG,$$   where   the   action   of  $\CG$  is  by
simultaneous  conjugation.
Since the characters we consider are unipotent and we assume $W$ to be irreducible,
it follows that $F$
fixes each $F$-stable  family pointwise.  Consequently, despite the fact that in
general  Lusztig considers a more complicated  set than $\CM(\CG)$ using an
automorphism  of $\CG$  induced by  $F$, we  do not  have to deal with this
more general situation here.

To describe Lusztig's Fourier transform matrix we
shall also require a sign $\Delta_{(x,\chi)}$ defined by Lusztig;
we will be more explicit
about its value when needed. For the moment we note that
\begin{itemize}
\item If $\bG$ is split, $\Delta_{(x,\chi)}=1$ except
for the ``exceptional'' families in types
$E_7$ and $E_8$ (those containing unipotent characters attached to irrational
representations of the Hecke algebra).
\item If $\bG$ is nonsplit, $\Delta_{(x,\chi)}$ depends only on $\CF$
(and not the particular $(x,\chi)$ considered);
we will therefore denote it by $\Delta_\CF$.
\end{itemize}

We  write  $\rho_{(x,\chi)}$  for  the  unipotent character parameterised by
$(x,\chi)$. There are two other bases of the space $\CC_\CF$ spanned by the
$\rho_{(x,\chi)}$ which play an important role.

\begin{definition}\label{def:mellin}({\cf} \cite[4.4]{DM2}).
Let $\CM'(\CG):=\{(x,y)\in\CG\times\CG\mid xy=yx\}/\CG$;
note  that  $\CM(\CG)$  and  $\CM'(\CG)$ have  the  same cardinality.
\begin{enumerate}
\item  For a representative  $(x,y)$ of $\CM'(\CG)$  (see above) define the
{\bf Mellin transform}
$$
\mu_{(x,y)}:=\sum_{\chi\in\Irr(C_\CG(x))}\chi(y)\rho_{(x,\chi)}.
$$
\item We define another basis $\{R_{(x,\chi)}\}_{(x,\chi)\in\CM(\CG)}$ of
$\CC_\CF$, the {\bf almost characters}, by the property that
$\mu_{(x,y)}=\Delta_\CF \sum_{\chi\in\Irr(C_\CG(x))}\chi(x)R_{(y,\chi)}$,
except in exceptional families of $E_7$ and $E_8$ where we have
$\CG=\BZ/2$ and if $\epsilon$ is the non-trivial character of $\CG$
the above formula must be modified to read:
$\mu_{(x,y)}=\epsilon(x)
\sum_{\chi\in\Irr(C_\CG(x))}\ol\chi(x)R_{(y,\chi)}$ (this takes into account
Lusztig's $\Delta_{(x,\chi)}$ in this case).
\end{enumerate}
\end{definition}
Definition (ii) above follows Lusztig \cite[4.24.1 and 13.6]{livre}. Note that in loc.\ cit.\
\begin{itemize}
\item If $c$ is the family in the group $W$ corresponding to $\CF$ (see the
beginning  of  section  \ref{ss:wff}),  Lusztig  parameterises the
preferred extensions $\tilde E$ of characters of
$c$ by some elements $x_{\tilde E}\in\CM(\CG)$.
\item
When  $(x',\chi')=x_{\tilde E}$  then $R_{(x',\chi')}$ is  equal to the
function $R_{\tilde E}$ of subsection \ref{ss:princbl}.
\end{itemize}

Clearly  knowledge of the values of all the irreducible characters in $\CF$
on a given class $C$ is equivalent to knowledge of the values on $C$ of the
almost characters or of the Mellin transforms.

The following result is proved by Shoji (see \cite[Th. 5.7]{Shoji1} and
\cite[Th. 3.2 and Th. 4.1]{Shoji2}). In rough language, it says that the \lf s
are, up to multiplication by a scalar which is a root of unity, equal to the
almost characters of $\GF$.

\begin{proposition}\label{prop:shoji}
As above, let $\CF$ be a family with associated group $\CG$.
The  \lf s in  $\CF$  may, just as  the irreducible characters, be
labelled  by pairs  $(x,\phi)\in\CM(\CG)$. Write  $\chi_{(x,\phi)}$ for the
\lf\  corresponding to $(x,\phi)\in\CM(\CG)$.  Then for
$(\bG,F)$ with connected centre and $p$ sufficiently large the
following  is true. For each $(x,\phi)\in\CM(\CG)$, there
is an algebraic number $\zeta_F$ of absolute value $1$ such that
$\chi_{(x,\phi)}=\zeta_F R_{(x,\phi)}$.
\end{proposition}
We have written $\zeta_F$ to emphasise the dependency on $F$; this number of
course {\it a priori} depends also on $(x,\phi)$. The condition ``sufficiently
large'' for $p$ above is ``almost good'', which means good for exceptional
groups, and no condition imposed for classical groups.

We shall now prove
\begin{theorem}\label{prop:projgamma-ac}
Maintain the assumption that is $\bG$ quasi-simple.
Suppose that we are in the setting of \ref{assume};
in particular $\CF$ is a family of unipotent characters of $\GF$ and $(u)$ is
its \wavefront. Then:
\begin{enumerate}
\item Any local system on $(u)$ is the (shifted) restriction of a
character sheaf lying in the unipotent Lusztig series.
\item For any $g\in A(u)$, we have $\Proj_\CF(\Gamma_{u_g})=\Delta_\CF
D(\mu_{(g,1)})$
except for the exceptional families of $E_7$ and $E_8$ where we have
$\Proj_\CF(\Gamma_{u_g})=\epsilon(g)D(\mu_{(g,1)})$.
\item If there is a pair $\ck\in\CI$ with support $(u)$ where $\CI$ is a
non-principal block with cuspidal datum $(\bL,\ci_\bL)$, the root of unity
$\zeta_F$ of Proposition \ref{prop:shoji} attached to $\ck$ is equal to
$\zeta_\CI(-1)^{\dim\supp\ci_\bL}\eta_\bL$.
\end{enumerate}
\end{theorem}
\begin{proof}
In the bijection of \cite[Th.\ 2.4 (b)]{RCS} between unipotent
character sheaves of $\CF$ and $\CM(\CG)$, the shifted restriction to the
class $(u)$ of the character sheaf with label $(1,\chi)$ is
the local system corresponding to $\chi$,
so that all local systems on $(u)$ appear in this bijection, whence (i).

Now consider  the case  when all  local systems  on $(u)$  are in the
principal block. Then by \cite[2.18]{Shoji1}, for all $F$-invariant
$E\in \Irr(W)$
we have $\chi_{E,\phi_E}=(-1)^{\rank\bG}R_{\wt E}$ when $\sigma_E$ in
subsection \ref{ss:cf} has
been chosen such that it defines the extension $\tilde E$,
in particular $\zeta_F=(-1)^{\rank\bG}$.
By Proposition \ref{prop:projunigamma}, and the definition of $\hat\ck$ above
Lemma \ref{lem:wf-cs}, if $E=\vp_{\ck}$
we have
$\Proj_\CF(\Gamma_\ck)=a_\ck\ve_\ck R_{\wt{E\otimes\ve}}$
(here we use the fact that
$\zeta_\CI=1$ for the principal block and that  $\dim\supp\ci_\bL=\rank\bG$
in Proposition \ref{prop:projunigamma}).
This can be written as $\Proj_\CF(\Gamma_\ck)=D(a_\ck R_{\wt E})$
since the Alvis-Curtis dual of $R_{\wt E}$ is
$R_{\wt E\otimes\wt\ve}=\ve_\ck R_{\wt{E\otimes\ve}}$ (see \cite[6.8.6]{livre}).
By \cite[Cor.~0.5]{FS}, if $\kappa=((u),\chi)$ then
$\wt E$ is parameterised by $(1,\chi)\in\CM(\CG)$, hence we have
$\Proj_\CF(\Gamma_\ck)=D(a_\ck R_{(1,\chi)})$.
By \ref{gammau} we have
$$
\Gamma_{u_g}=a_u\inv\sum_{\psi\in\Irr(A(u))}\ol{\psi(g)}\Gamma_{((u),\psi)}.
$$
Applying $\Proj_\CF$ to both sides of this relation, and using the above
we obtain
\begin{equation}\label{proj Gamma}
\Proj_\CF(\Gamma_{u_g})=D(\sum_{\psi\in\Irr(A(u))}
\ol{\psi(g)}R_{(1,\psi)}).
\end{equation}

By \ref{def:mellin} (ii), the right side of \eqref{proj Gamma} is equal to
$\Delta_\CF D(\mu_{(g,1)})$
except in the exceptional families of $E_7$ and $E_8$ where it is
$\epsilon(g)D(\mu_{(g,1)})$.

In  the case where  there is a  pair $\ck=((u),\eta)$ not  in the principal
block  but  in  a  block  $\CI$  with  cuspidal  data  $(\bL,\ci_\bL)$ then
Proposition \ref{prop:projunigamma}, using Corollary \ref{D(chi_E)} becomes
$\Proj_\CF(\Gamma_\ck)=a_\ck\zeta_\CI\inv(-1)^{\dim\supp\ci_\bL}   \eta_\bL
D(\chi_{E,\phi_E})$  where $A_E$ is determined  by $E=\varphi_\ck$. As explained
in  the  beginning  of  the  proof,  the  character  sheaf  $A_E$ has label
$(1,\eta)$. Now applying Proposition \ref{prop:shoji} we get
$\Proj_\CF(\Gamma_\ck)=a_\ck\zeta_\CI\inv(-1)^{\dim\supp\ci_\bL}   \eta_\bL
\zeta_F D(R_{(1,\eta)})$, and writing
$\zeta=\zeta_\CI\inv\zeta_F(-1)^{\dim\supp\ci_\bL}\eta_\bL$, we obtain
$$
\begin{aligned}
\Proj_\CF(\Gamma_{u_g})&=D(\sum_{\overset{\psi\in\Irr(A(u))}{\psi\neq\eta}}
\ol{\psi(g)}R_{(1,\psi)})+\zeta\ol{\eta(g)}
D(R_{(1,\eta)})\\
&=D(\sum_{\psi\in\Irr(A(u))}{\psi(g\inv)}R_{(1,\psi)})
+(\zeta-1)\eta(g\inv)D(R_{(1,\eta)})\\
&=\Delta_\CF D(\mu_{(g,1)})+(\zeta-1)\eta(g\inv)D(R_{(1,\eta)}),\\
\end{aligned}
$$
where $\mu$ is the Mellin transform as explained in the case where all local
systems were in the principal block (here we use the fact that we are not in an
exceptional family).

But  the  left  hand  side  is  a  proper  character  of $\GF$, and since
$\mu_{(g,1)}$  is a proper character and $R_{(1,\eta)}$ has rational
coefficients in the basis of  characters, it is  immediate that $\zeta=\pm
1$.  Case by  case inspection  shows that  if there  is a  single pair with
support $(u)$ not in the principal block, then
$A(u)\in\{\Sym_3,\Sym_4,\Sym_5\}$  or $A(u)\simeq (\BZ/2\BZ)^k$ with $k\geq
2$ (that is, since we assume $\bG$ adjoint, $A(u)\not\simeq\BZ/2\BZ$);
indeed, for adjoint exceptional groups, both local systems are in the
principal block when $\CG=A(u)=\BZ/2\BZ$; for adjoint classical groups, if
$\CG=A(u)=\BZ/2\BZ$, one checks that the only  pair $(x,\chi)\in\CM(\CG)$
which parameterises a local system which is not in the principal block has
$x\neq 1$ hence by \cite[Th.\ 2.4 (b)]{RCS}
does not correspond to a local system on $(u)$.

In the  non-abelian case $A(u)\in\{\Sym_3,\Sym_4,\Sym_5\}$,
the formula  in \cite{DM1}  for the Fourier
transform    shows   that   the   coefficient   of   $\rho_{(y,\chi)}$   in
$R_{(1,\eta)}$  is non-zero for any $\chi$ if $y$ is a $3$-cycle. If
$\zeta=-1$,  not all of  the negative terms  can be cancelled  in the right
side  of  the  above  equation.

In  the  abelian  case $A(u)=(\BZ/2\BZ)^k$ with $k>1$,  if  $\zeta=-1$ the
coefficient of $\rho_{(y,\chi)}$ in $R_{(1,\eta)}$ has a denominator
$2^{k-1}$ with $k$ as above. Hence in both cases we conclude that $\zeta=1$
and $\Proj_\CF(\Gamma_{u_g})=D(\mu_{(g,1)})$.
\end{proof}
The statement (ii) of Theorem \ref{prop:projgamma-ac}  proves implicitly
that $\Delta_\CF D(\mu_(g,1))$ (resp.\ $\epsilon(g)D(\mu(g,1))$) is an actual
character. We remark that this is also a consequence of the result
\cite[6.20]{livre} of Lusztig.

\subsection{Values of unipotent characters at unipotent elements, and concluding remarks}
In the three groups covered by Remark \ref{lem:assume}, as well as in groups
of type $E_6$, $\lexp 2E_6$  and $E_7$, all the
families  of unipotent characters other than  one family $\CF$ satisfying the
assumptions of Theorem \ref{prop:projgamma-ac}
are in  the principal  series, i.e.  are of  the form
$(\Qlbar,c)$.  Hence their values at unipotent  elements are given by Green
functions, which may be taken as known by Lusztig's algorithm.

Thus  to determine all values of  all unipotent characters in $\Irr(\GF)$
at  unipotent  elements,  it  suffices  to  determine  the  values  of  the
characters  in  $\CF$.  But  by  Proposition  \ref{prop:shoji}, for this it
suffices  to determine the root  of unity $\zeta_F$, since  the values of the
\lf  s $\chi_{(x,\phi)}$  ($(x,\phi)\in\CM(\CF)$) are  known. Hence  in these
cases,  Theorem \ref{prop:projgamma-ac}(ii) completes the determination
of the values of unipotent characters at unipotent classes. The constant
$\zeta_F$ was already determined by Lusztig \cite[8.6]{CV} and implicitely
by Kawanaka \cite[4.2.2]{K2}, using different methods which each yield
different restrictions on the applicable values of $p$ and $q$.

We  remark finally  that empirically,  the values  of the Mellin transforms
$\mu_{(x,y)}$ at unipotent classes appear to be ``simpler'' than the values
of  the irreducible characters, when expressed  as polynomials in $q$. This
remark is based on computations in the exceptional groups, and currently we
have no theoretical justification for it.

\appendix

\section{Classification of cuspidal character sheaves}\label{s:classcs}

Let  $\bG$ be a  connected reductive algebraic  group over an algebraically
closed field of characteristic $p$ (we allow $p=0$).

We  summarise  here  the  classification  of  cuspidal character sheaves on
$\bG$,  which is essentially due to  Lusztig (see \cite{CS}). In particular
we  give a list,  conveniently arranged, of  cuspidal character sheaves for
each isogeny type of quasi-simple group.

As explained in \ref{ss:cs}, a cuspidal character sheaf $A$ is the perverse
extension  of an  irreducible cuspidal  local system  whose support  is the
inverse  image in  $\bG$ of  a conjugacy  class of $\bG/Z^\circ_\bG$, where
$Z_\bG$  is the centre  of $\bG$. If  $x$ is in  the support of $A$ then by
\cite[2.8]{ICC}   the  group  $C_\bG^0(x)/Z^\circ_\bG$   is  unipotent;  in
particular  if $x_sx_u$ is the Jordan  decomposition of $x$ the semi-simple
part   $x_s$  is  isolated.   The  class  of   $x_u$  is  distinguished  in
$C_\bG^0(x_s)$ and $Z_\bG/Z^\circ_\bG$ injects into
$A_\bG(x)=C_\bG(x)/C_\bG^0(x)$.  By  the  cleanness  of  cuspidal character
sheaves (see \cite{Lu12}) $A$ vanishes outside the support of $\CE$.

As  explained in \ref{ss:wff},  a cuspidal character  sheaf has a ``Lusztig
series'',  parameterised by the conjugacy class of some semi-simple element
$s$ of the Langlands dual group $\bG^*$ to $\bG$; the sheaf further defines
a  label  inside  a  family  attached  to the group $W'(s)=W_{\bG^*}(s)$.
Specifically,  the  group  $W'(s)$  is  of the form $W(s)\rtimes\Omega$ where
$W(s)=W(C_{\bG^*}^0(s))$  and where $\Omega\simeq A_{\bG^*}(s)$; a family of
$W'(s)$ with a label attached to a cuspidal character sheaf is determined by
an  $\Omega$-stable family of $W(s)$; if  the small finite group giving rise
to  the labels of this  last family is $\CG$,  the group giving rise to the
labels of the family of $W'(s)$ is $\CG\rtimes\Omega$ (see
\cite[17.1--17.8]{CS}).  There is at most one family in a given group $W(s)$
arising  as above  (\cite[Chapter 8]{livre}).  The labels  for the cuspidal
character   sheaves  in  the  family  are  pairs  $(x,\chi)$  taken  up  to
$\CG\rtimes\Omega$-conjugacy    where    $x\in\CG\rtimes\Omega$   and
$\chi\in\Irr(C_{\CG\rtimes\Omega}(x))$. We will call cuspidal labels the
labels which can be labels for cuspidal character sheaves.

To  each  irreducible  local  system  on  the  class of $x$ is associated a
central  character, coming from the  action of  $Z_\bG$ on  fibres;
it is also the restriction to $Z_\bG$ of the character of $A_\bG(su)$
associated to the local system.

We will use the following facts

\begin{itemize}
\item  The cuspidal character sheaves on a direct product of groups are the
external  tensor product  of cuspidal  character sheaves  on each component
(\cite[17.11]{CS}).
\item  The cuspidal character  sheaves on $\bG$  are obtained from those on
$\bG/Z_\bG^\circ$ by inverse image and tensoring by a local system $\CS$ of
rank  1 on  $\bG$ which  is the  inverse image  of a  Kummer system  on the
abelianisation  $\bG/[\bG,\bG]$. The effect on the Lusztig series of tensoring
by  $\CS$ is to multiply the semi-simple  element of $\bG^*$ by the central
element corresponding to $\CS$ (see Lemma \ref{lem:series} or \cite[17.9, 17.10]{CS}).
\item  Let $\tilde\bG\xrightarrow\pi\bG$  be a  surjective morphism  with a
finite  central kernel;  then, given  an irreducible  cuspidal local system
$\CS$  on $\tilde\bG$,  the direct  image $\pi_*\CE$  is a  local system if
$\CE$  is $\ker\pi$-invariant (equivalently  its central character vanishes
on  $\ker\pi$);  in  this  latter  case,  the  components of $\pi_*\CE$ are
irreducible  cuspidal local systems and all cuspidal local systems on $\bG$
are  obtained this way \cite[2.10]{ICC}. If $x$ is in the support of $\CE$,
the  direct image $\pi_*$ corresponds to the induction through the natural
morphism $A_{\tilde\bG}(x)\to A_\bG(\pi(x))$. If
$\tilde\bG^*\xleftarrow{\check\pi}  \bG^*$ is the dual map, and $\CE'$ is a
component  of $\pi_*  \CE$ with  Lusztig series  $s$, then the Lusztig
series of $\CE$ is $\check\pi(s)$ \cite[17.16]{CS}.
\end{itemize}

The  above  facts  in  principle   reduce the classification to the case of
quasi-simple  and simply connected groups. But in practice the list is much
easier  to use if  the classification is given  for each isogeny type  of quasi-simple group,
and this is what we will do.

We assume now $\bG$ quasi-simple.  We use the additional facts
\begin{itemize}
\item  The  central  character  associated  to  a  cuspidal character sheaf
determines the part ``element of $\Omega$'' of its label \cite[23.0]{CS}.
\item For $z\in Z_\bG$, the associated translation operator $t_z$ preserves
the  Lusztig series \cite[17.17.2]{CS} and acts  on the part ``character of
$\Omega$''  of the label as follows: if  the label of $A$ is $(x,\chi)$ and
the  central character of $A$ is  faithful, and $\sigma_z$ is the character
of   $\Omega$  determined  by   $z$,  then  the   label  of  $t^*_z  A$  is
$(x,\chi\sigma_z)$ \cite[23.1 (c)]{CS}.
\end{itemize}

To describe cuspidal character sheaves whose support contains $x$,
we  give in order a description of $x_s$, of $x_u$ and of $A_\bG(x)$ (often
it  is sufficient for  this to describe  $C_\bG(x_s)$, $A_\bG(x_s)$ and the
class  of  $x_u$  in  $C_\bG(x_s)^0$),  and finally describe the associated
character of $A_\bG(x)$.
Note  that the class of  $x_u$ is determined by  the class of $x_s$ and the
fact  that the class of $x$ supports a cuspidal local system if and only if
the  class  of  $x_u$  in  $C_G(x_s)^0$ supports a
cuspidal local system (see \cite[2.10]{ICC}).

We also need the following
\begin{lemma}\label{G tilde to G}
Let  $\pi:\tilde \bG\rightarrow  \bG$ be  surjective morphism  with central
kernel;  let $\tilde  x\in \tilde\bG$  and let  $x=\pi(\tilde x)$; there is
then an exact sequence $$1\rightarrow\ker\pi/(\ker\pi\cap
C_{\tilde\bG}(\tilde        x)^\circ)\rightarrow       A_{\tilde\bG}(\tilde
x)\xrightarrow{\overline\pi}A_\bG(x)\xrightarrow\eta\ker\pi,$$        where
$\overline\pi$  is induced by $\pi$ and where, if $g\in C_\bG(x)$ is both a
representative  of  $\overline  g\in  A_\bG(x)$  and  the  image of $\tilde
g\in\tilde\bG$,  we define $\eta(\overline  g)=z$ where $z$  is the element
defined by $\lexp{\tilde g}{\tilde x}=\tilde xz$. The map $\overline\pi$ is
surjective when $x$ is unipotent.

If  in addition the $\bG$- conjugacy class  of $x$ affords a cuspidal local
system then $(\ker\pi\cap C_{\tilde\bG}(\tilde        x)^\circ)=1$,
\end{lemma}
\begin{proof}
We  first  observe  that  the  formula  for  $\eta(\overline  g)$  gives  a
well-defined element $z\in \ker\pi$, whence a morphism
$C_\bG(x)\rightarrow\ker\pi$    compatible    with    the    quotient    by
$C_\bG(x)^\circ$   and  whose  kernel   is  by  definition   the  image  of
$C_{\tilde\bG}(\tilde  x)$, whence an exact sequence $$A_{\tilde\bG}(\tilde
x)\xrightarrow{\overline\pi}A_\bG(x)\xrightarrow\eta\ker\pi.$$           As
$\pi(C_{\tilde\bG}(\tilde    x)^\circ)=C_\bG(x)^\circ$,   the   kernel   of
$\overline\pi$ is $\ker\pi/\ker\pi\cap C_{\tilde\bG}(\tilde x)^\circ$.
If $x$ is unipotent, taking the semi-simple parts on both sides of the
equality $\lexp{\tilde g}{\tilde x}=\tilde xz$ we find $z=1$.

When  the class  of $\tilde  x$ affords  a cuspidal  local system the group
$C_{\tilde\bG}(\tilde x)^\circ$    is   unipotent   which   implies   that
$\ker\pi\cap C_{\tilde\bG}(\tilde x)^\circ=1$, whence the result.
\end{proof}

We  use Lemma \ref{G tilde to G} as  follows: given a local system $\CE$ on
the  class of  $\tilde x$  lifting a  local system  on the class of $x$, or
equivalently   a  representation  of  $A_{\tilde\bG}(\tilde  x)$  factoring
through  a representation $\rho$ of $\overline\pi(A_{\tilde\bG}(\tilde x))$,
then $\pi_*\CE$ corresponds to the induced of $\rho$ to $A_\bG(x)$. We will
use two special cases: if the image of $\eta$ is trivial, then $\pi_*(\CE)$
is  irreducible; if the  image of $\eta$  is of prime  cardinality $r$ then
either   $\pi_*\CE$  is  irreducible  or  has  $r$  irreducible  components
(depending whether $\rho$ is invariant or not by $A_\bG(x)$).


\section{Classical groups}\label{s:classcl}
In  a classical Weyl group, the groups $\CG$ attached to families are of
the  form $(\BZ/2\BZ)^n$, and the labels in  a family are in bijection with
Lusztig's  ``symbols'' (pairs  of increasing  sequences of natural integers
taken  up to shift, see \cite[Chapter  4]{livre}); we will call cuspidal a
symbol corresponding to a cuspidal label.

A  useful preliminary is  the description of  cuspidal symbols for $W(B_n)$
and  $W(D_n)$. For  $W(B_n)$ where  $n=d^2+d$ the  only cuspidal  symbol is
$(\{0,1,\ldots,2d\},\emptyset)$;    It    corresponds    to    the    label
$(x,\chi)=((-1,1,-1,1,\ldots),(-1,-1,-1,-1,\ldots))$ in $\CM(\CG)$ where
$\CG=(\BZ/2\BZ)^d$.  For  $W(D_n)$  with  $n=(d+1)^2$  the only cuspidal
symbol is $(\{0,1,\ldots,2d+1\},\emptyset)$; $\CG$ and $(x,\chi)$ are as
above.

\subsection*{Type $A_{n-1}$}
The  only group of type $A_{n-1}$  which affords cuspidal character sheaves
is  $\bG=\SL_n$ when  $p$ does  not divide  $n$. There  are $n\phi(n)$ such
sheaves,  parameterised  by  the  pairs  $(z,\chi)$  where $z\in Z_\bG$ and
$\chi$  is  an  injective  character  of $\chi\in\Irr(Z_\bG)$. The cuspidal
local  system associated  to $(z,\chi)$  has support  $zC$ where $C$ is the
regular  unipotent class, and $\chi$ can  be identified to the character of
$A_\bG(x)=Z_\bG$  where  $x\in  zC$  associated  to the local system. These
character sheaves are all in the Lusztig series defined by a quasi-isolated
$s\in\bG^*$  such  that  $W'(s)$  contains  a  Coxeter  element of
$W(\bG^*)$;  we have  $W(s)=1$ and  $\Omega=W'(s)$ is  in bijection
with $\Irr(Z_\bG)$. Then $z\in Z_\bG$ corresponds to $\zeta\in\Irr(\Omega)$
and  $\chi$ to  a generator  $x\in\Omega$, and  $(x,\zeta)$ is the label in
$\CM(\Omega)$ of the considered sheaf (see \cite[18.5]{CS}).

\subsection*{Types $B_n$ and $C_n$ for $p=2$}
The  description is  the same  in both  cases; there are cuspidal character
sheaves  only when $n/2$ is a triangular number, in which case there is one
\cite[22.2]{CS}. Its support is the class of a unipotent $x_u$ of Jordan type
(in   $\Sp_{2n}$,  see\cite[2.7]{LS})  given   by  $(4,8,12,\ldots)$;  this
partition  has $d$ parts where  $n=d(d+1)$. We have $A_\bG(x_u)=(\BZ/2\BZ)^d$
and the local system corresponds to the character $(-1,1,-1,1,\ldots)$. The
symbol  (in the sense  of \cite[2.7 and  1.2e]{LS}) of the  local system is
$(\{0,4,\ldots,4d\},\emptyset)$   when   $d$   is   odd   and  $(\emptyset,
\{2,\ldots,2+4(d-1)\})$  when $d$ is  even. The Lusztig  series is $\Qlbar$
(the  unipotent series)  and the  label is  the cuspidal symbol of $W(B_n)$
(see. \cite[22.4, 22.6]{CS}).

\subsection*{Type $C_n$ for $p\ne 2$}
It  is useful to  first describe the  cuspidal local systems with unipotent
support.  For $C_n^\ad$  there exists  such a  system when  $n$ is  an even
triangular  number. For $C_n^\sco$ there  is also such a  system for $n$ an
odd  triangular number, which  has a nontrivial  central character. In each
case  the support is described by the Jordan type $(2,4,6,\ldots,2d)$ where
$n=d(d+1)/2$  and $A_\bG(x_u)=(\BZ/2\BZ)^{d-1}$; the local system corresponds
to  the character $(-1,1,-1,1,\ldots)$ of $A_\bG(x_u)$  and the symbol of the
local  system  is  $(\{0,2,4,\ldots,2d\},\emptyset)$  if  $d$  is  even and
$(\emptyset,   \{1,3,5,\ldots,2d+1\})$   otherwise   (see   \cite[11.5  and
12.4]{ICC}).

\subsubsection*{Case $C_n^\ad=\PSp_{2n}$}
$\PSp_{2n}$  affords cuspidal local systems if and  only if $n$ is even and
of  the form $n=N_1+N_2$  where $N_1$ and  $N_2$ are triangular numbers (of
the  same parity).  If $N_1\ne  N_2$ there  exists a  unique cuspidal local
system with support the class of $x$ where $C_\bG(x_s)\simeq
\Sp_{2N_1}\times^{\BZ/2\BZ}\Sp_{2N_2}$. If $N_1=N_2$ there are two cuspidal
local systems on the class of $x$ where
$C_\bG(x_s)\simeq(\Sp_{2N_1}\times^{\BZ/2\BZ}\Sp_{2N_1})\rtimes    \BZ/2\BZ$;
the  non-trivial element of $A_\bG(x_s)$  exchanges the two $\Sp$ components,
and  the two local  systems correspond to  the two characters of $A_\bG(x_s)$
(see  \cite[23.2(a)]{CS}; by \cite[2.10]{ICC}  the local system corresponds
to  a  unipotently  supported  cuspidal  local  system  on $C_\bG(x_s)^0$, of
central  character $\Id\boxtimes\Id$ when $N_1$ and  $N_2$ are even, and of
central  character $\ve\boxtimes\ve$  when $N_1$  and $N_2$
are  odd). If $N_1=t(t+1)/2$ and $N_2=r(r+1)/2$  with $t,r\geq 0$, then the
Lusztig    series   is   given   by   $s\in\Spin_{2n+1}$   such   that
$C_{\bG^*}(s)\simeq  \Spin_{2A}\times^{\BZ/2\BZ}\Spin_{4B+1}$ where if
$t\not\equiv  r\pmod 2$  we have  $4A=(t+r+1)^2$ and  $8B+1=(t-r)^2$ and if
$t\equiv r\pmod 2$ we have $4A=(t-r)^2$ and $8B=(t+r+2)(t+r)$ (see \cite[p.
976]{remarks}  and \cite[23.16]{CS}). The above  describes a unique element
$s$  if $t\ne  r$ and  two (central)  elements if  $t=r$. The label is
given  by  the  cuspidal  symbol  of the group $W(s)\simeq W(D_A)\times
W(B_{2B})$ (with the convention that $W(D_1)$ is the trivial group).

\subsubsection*{Case $C_n^\sco=\Sp_{2n}$}
Again  we  must  have  $n=N_1+N_2$  where  $N_1$  and  $N_2$ are triangular
numbers.  If $n$  is even  and $N_1\ne  N_2$ there  are two  local cuspidal
systems, inverse images of the one in $C_n^\ad$, where the semi-simple part
of  the support is respectively $x_s$ and  $x_sz$, where $z$ is the non-trivial
element  of $Z\bG$. If $N_1=N_2$ the element  $x_s$ is conjugate to $x_sz$, and
$C_\bG(x_s)$  is  connected,  which  leads  to  a  single  local  system.  If
$N_1\not\equiv N_2 \pmod 2$ there are also two cuspidal local systems, with
a nontrivial central character (see \cite[23.2 (b)]{CS}). In every case one
can  index the cuspidal local systems by ordered pairs $(N_1,N_2)$; we have
$C_\bG(x_s)\simeq\Sp_{2N_1}\times\Sp_{2N_2}$. If $t,r,A,B$ are defined by the
same  formula as in the previous subsection, the Lusztig series is given by
an element $s$ such that
$C_{\bG^*}(s)\simeq\Orth_{2A}\times\SO_{4B+1}$.  The  group  $\Omega$  is
thus  $\BZ/2\BZ$ excepted if $N_1=N_2$ (the element of $\Omega$ part of the
label  corresponds to the semisimple part $x_s$ or $x_sz$ of the support of the
local  system). The  part ``in  $W(s)$'' of  the label  is the cuspidal
symbol of the group $W(s)\simeq W(D_A)\times W(B_{2B})$.

\subsection*{Type $B_n$ for $p\ne 2$}
It  is useful to  first describe the  cuspidal local systems with unipotent
support.  For $B_n^\ad$ there  is a (unique)  such system when  $2n+1$ is a
square  (see  \cite[13.4]{ICC}).  For  $B_n^\sco$,  there  is an additional
system  for $2n+1$ a triangular number, with a nontrivial central character
\cite[14.6]{ICC}.  In  the  first  case  the  Jordan type of the support is
$(1,3,5,\ldots)$  and if $d$ is the number  of parts of this partition then
$A_{\SO}(x_u)=(\BZ/2\BZ)^{d-1}$    \cite[10.6]{ICC};    the    local   system
corresponds  to  the  character  $(-1,1,-1,\ldots)$  of this group. We have
$n=d^2$  and the  symbol (in  the sense  of \cite[13.4]{ICC})  of the local
system  is  $(\{0,2,4,\ldots,2d-2\},\emptyset)$.  In  the  second  case the
Jordan  type of  the support  is $(1,5,9,\ldots)$  or $(3,7,11,\ldots)$. We
have $A_\bG(x_u)=(\BZ/2\BZ)^d$ where $d$ is the number of parts of the Jordan
type (see \cite[10.6]{ICC} and Lemma \ref{G tilde to G}).

\subsubsection*{Case $B_n^\ad=\SO_{2n+1}$}
We  must have $2n+1=N_1+N_2$ where $N_1$ is an even square and $N_2$ an odd
square.  There is  one cuspidal  local system,  unipotently supported, when
$N_1=0$, and two cuspidal local systems if $N_1\ne 0$ (see
\cite[23.2(c)]{CS});   in   this   last   case   we   have  $C_\bG(x_s)\simeq
\Orth_{N_1}\times  \SO_{N_2}$. Let  $N_1=r^2$ and  $N_2=t^2$; then the Lusztig
series is defined by $s$ of centraliser
$\Sp_{((r+t)^2-1)/2}\times\Sp_{((r-t)^2-1)/2}$.

\subsubsection*{Case $B_n^\sco=\Spin_{2n+1}$}
Concerning systems coming from $B_n^\ad$ (with a trivial central character)
if  $N_1$ and $N_2$ as  above are distinct and  both nonzero, the preimages
$x_s$  and $x_sz$ are conjugate  and we have a  single local system. If $N_1=0$
there  are two systems, one with unipotent support, the other translated by
the  non-trivial  element  of  the  centre.  When  the central character is
nontrivial  (see \cite[23.2  (e)]{CS}), we  must have  $2n+1=N_1+N_2$ where
$N_1$  is an  even triangular  numbers and  $N_2$ an odd triangular number.
Each  such pair gives rise  to two cuspidal complexes.  When $N_1\ne 0$ and
$N_2\ne  1$,  they  share  the  same  support,  the  class  of  $x$  where
$C_\bG(x_s)\simeq\Spin_{N_1}\times^{\BZ/2\BZ}\Spin_{N_2}$  ($x_s$ and  $x_sz$ are
conjugate);  $\Spin_{N_1}$ has two  central characters which  restrict to a
nontrivial character of $Z_\bG$, corresponding to the two local systems. If
$N_2=1$  we have $C_\bG(x_s)\simeq\Spin_{N_1}$ and if  $N_1=0$ one of them is
unipotently  supported and the other  translated by the non-trivial element
of the centre. Let $\{r,t\}$ be positive integers such that
$\{N_1,N_2\}=\{r(r+1)/2,t(t+1)/2\}$.  If $r$ have $t$ different parity, let
$t$  be the even one. Then we must have $r\equiv t+1\pmod 4$ (for $N_1+N_2$
to  be odd), and we  let $A=(r+t+3)(r+t-1)/16$ and $B=(r-t-1)(r-t+1)/8$. If
$r$  and $t$ have same parity we must  have $r\equiv t+2\pmod 4$ and we let
$A=(r-t-2)(r-t+2)/16$  and  $B=(r+t)(r+t+2)/8$.  The  Lusztig  series  (see
\cite[1.11]{remarks})   is  then  defined  by  $s\in\bG^*$  such  that
$C_{\bG^*}(s)\simeq((\Sp_{2A}\times\Gl_B\times
\Sp_{2A})/(\BZ/2\BZ))\rtimes\BZ/2\BZ$   where  the  nontrivial  element  of
$A_{\bG^*}(s)$  exchanges  the  two  $\Sp$  components and induces the
transpose inverse automorphism of the $\Gl$ component.

\subsection*{Type $D_n$ for  $p=2$}
There is at most one cuspidal character sheaf,  and it occurs when $n$ is
an even square \cite[22.3]{CS}.  The support is unipotent of
Jordan type  $(2,6,10,\ldots,4d-2)$; this  partition  has
$d$  parts   where  $n=d^2$  (see  \cite[3.3  and   1.2e]{LS});  we have
$A_\bG(x_u)\simeq(\BZ/2\BZ)^{d-1}$ and the symbol in the  sense  of
{\sl loc.  cit.}  of the local system  is $(\{0,4,\ldots,4d\},\emptyset)$.
The Lusztig series is $\Qlbar$ (unipotent series) and the label is the
cuspidal symbol of $W(D_n)$ (see \cite[22.7]{CS}).

\subsection*{Type $D_n$ for $p\ne 2$}
It  is useful to  first describe the  cuspidal local systems with unipotent
support. For $D_n^\ad$ there is such a system exactly when $2n$ is a square
and  $n/2$ is even. For $\SO_{2n}$, there an additional system when $2n$ is
a  square  and  $n/2$  is  odd,  with  a  nontrivial central character. For
$\Spin_{2n}=D_n^\sco$  there  is  additionally  a  system  when  $2n$  is a
triangular number, attached to each of the two central characters which are
nontrivial  on an  element of  the kernel  of $\Spin_{2n}\to  \SO_{2n}$. If
$2n=d^2$  the support $x_u$  has Jordan type  $(1,3,5,\ldots, 2d-1)$, we have
$A_{\SO}(x_u)=(\BZ/2\BZ)^{d-1}$  (\cite[10.6]{ICC})  and  the  symbol  in the
sense   of  \cite[13.4]{ICC}   of  the   local  system  is  $(0,2,4,\ldots,
2d-2,\emptyset)$. If $4n=d(d+1)$ with $d$ odd (resp.\ even), the Jordan type
of  the  support  is  $(1,5,9,\ldots,2d-1)$ (resp.\ $(3,7,11,\ldots,2d-1)$),
(see  \cite[4.9]{LS}).  In  these  latter  cases  $A_\bG(x_u)$  is a nonsplit
central extension by $\BZ/2$ of $A_{\SO}(x_u)=(\BZ/2\BZ)^{[\frac{d-1}2]}$ and
the  two cuspidal local systems correspond  to the two characters of degree
$2^{[\frac{d-1}4]}$   of  this  group  (\cite[14.3,  14.4]{ICC})(note  that
$d\equiv  0$ or $3$ $\pmod 4$, thus $[\frac{d-1}2]$ is odd). Note also that
when  $n$ is even  the unipotently supported  cuspidal local system are the
preimage of those of $\frac12\Spin_{2n}$.

\subsubsection*{Case $D_n^\ad=\PSO_{2n}$}

$\PSO_{2n}$ affords cuspidal local systems if and only if $n$ is a multiple
of  $4$ and  $2n=N_1+N_2$ where  $(N_1,N_2)$ is  an unordered  pair of even
squares.  If  $N_1$  or  $N_2$  is  zero  it  affords  one cuspidal system,
otherwise  if  $N_1$  and  $N_2$  are  nonzero  and  distinct there are two
cuspidal  systems. If  $N_1=N_2\neq 0$  there are  4 cuspidal  systems (see
\cite[23.2 (c)]{CS}). These systems all have the same support, the class of
$x$  where $C_\bG(x_s)^\circ\simeq  \SO_{N_1}\times^{\BZ/2\BZ}\SO_{N_2}$. If
$N_1$  or $N_2$  is zero  then $x_s=1$,  otherwise if  $N_1\neq N_2$  we have
$A_\bG(x_s)=\BZ/2\BZ$,  acting by the outer automorphism on both factors, and
if $N_1=N_2\neq 0$ we have $A_\bG(x_s)\simeq\BZ/2\BZ\times\BZ/2\BZ$ where the
second   factor  $\BZ/2\BZ$   exchanges  the   two  $SO_{N_1}$  factors  of
$C_\bG(x_s)^\circ$.

The  Lusztig series  is defined  by $s$  such that $C_{\bG^*}(s)$
affords  the double cover  $\Spin_{2a^2}\times\Spin_{2b^2}$ where $a\geq 0$
and   $b\geq  0$  are   given  by  $N_1=(a+b)^2$   and  $N_2=(a-b)^2$  (see
\cite[1.12]{remarks}  or \cite[23.19 (c$_2$)]{CS}). There  are as many such
semi-simple  classes as cuspidal  local systems, and  each system lies in a
different  series. The label of a cuspidal system is the cuspidal symbol of
the group $W(D_{a^2})\times W(D_{b^2})$.

\subsubsection*{Case $\SO_{2n}$}

$\SO_{2n}$ affords cuspidal local systems if and only if $2n=N_1+N_2$ where
$N_1$  and $N_2$  are even  squares. If  $n\equiv 0  \pmod 4$  the cuspidal
systems  are a preimage from $D_n^\ad$ and have a trivial central character
(using  Lemma  \ref{G  tilde  to  G}).  If  $n\equiv  2\pmod 4$ they have a
nontrivial central character (see \cite[23.2 (d) and 23.19 (d)]{CS}).

A description which covers both cases is as follows:
for each ordered pair $(N_1,N_2)$
there  is a cuspidal system if $N_1=0$  or $N_2=0$ and two cuspidal systems
otherwise. The supports are the class of $x$ such that
$C_\bG(x_s)^\circ\simeq\SO_{N_1}\times\SO_{N_2}$:  there are two such classes
if  $N_1\neq N_2$ (which is always the case if $n\equiv 2\pmod 4$), that we
parameterise repectively by $(N_1,N_2)$ and $(N_2,N_1)$, and only one class
if  $N_1=N_2$. If $N_1$ and $N_2$  are nonzero $A_\bG(x_s)$ is $\BZ/2$ acting
by  the simultaneous exterior automorphism  of both components. If $N_1\neq
N_2$  the translation by  the centre exchanges  $N_1$ and $N_2$ (exchanging
the supports). If $N_1=N_2$, the centre acts trivially (then $n\equiv0\pmod
4$).  We have  $x_s=1$ if  $N_1=0$ and  $x_s$ a  nontrivial central  element if
$N_2=0$.   The  Lusztig   series  is   described  by   $s$  such  that
$C_{\bG^*}^\circ(s)\simeq  \SO_{2a^2}\times\SO_{2b^2}$  where  $a$ and
$b$  are defined by the  same formulae as in  the $\PSO_{2n}$ case excepted
that  $b$ dmay have an arbitrary sign. Note that if $N\equiv 2\pmod 4$, $a$
and  $b$ must be odd. There is two such classes excepted if $a=\pm b$, that
is $N_1=0$ or $N_2=0$. If $a$ and $b$ are nonzero we have
$A_{\bG^*}(s)\simeq  \BZ/2\BZ$  acting  by  the  simultaneous exterior
automorphism of $\SO_{2a^2}$ and $\SO_{2b^2}$.

\subsubsection*{Case $\frac12\Spin_{2n}$}

There  exist cuspidal  systems with  a non-trivial  central character (see
\cite[23.2  (f)]{CS})  if  and  only  if  $n\geq  6$  is  odd,  of the form
$2n=N_1+N_2$  where $N_1$ and  $N_2$ are even  triangular numbers. For each
such  non ordered pair $(N_1,N_2)$ there  are two cuspidal local systems if
$N_1\neq  N_2$ and four  otherwise (see \cite[23.2  (f)]{CS}). If $N_1$ and
$N_2$  are non zero  the systems have  all support the  class of $x$ where
$C_\bG(x_s)^\circ$ is isogenous to $\SO_{N_1}\times\SO_{N_2}$. If $N_1=0$ one
of the two systems has unipotent support and the other is translated by the
nontrivial  element of the centre. In any  case the action of the centre is
free on the 2 or 4 systems. Let $N_1=\frac{r(r+1)}2$ and
$N_2=\frac{t(t+1)}2$ with $r,t\geq 0$. Then $r$ and $t$ are equal to $0$ or
$3 \pmod 4$.

If  $r\not\equiv t\pmod 4$  the Lusztig series  is defined by $s$ such
that  $C_{\bG^*}(s)^\circ$  is  isogenous  to
$$
\SO_{\frac{(r+t+1)^2}8}
\times\Gl_{\frac{(r-t)^2-1}8}\times\SO_{\frac{(r+t+1)^2}8};
$$
if  $r\equiv
t\pmod   4$  the   Lusztig  series   is  defined   by  $s$  such  that
$C_{\bG^*}(s)^\circ$   is  isogenous  to
$$
\SO_{\frac{(r-t)^2}8}\times
\Gl_{\frac{(r+t)(r+t+2)}8}\times\SO_{\frac{(r-t)^2}8}
$$
(see   \cite[23.19
(f)]{CS}). There is a unique such class, with
$A_{\bG^*}(s)\simeq\BZ/2\BZ$   acting   by   interchanging   the   two
$\SO$ components.

The  other cuspidal local systems have a trivial central character and come
from  $\PSO_{2n}$ (we use here Lemma \ref{G tilde to G}). Then $2n=N_1+N_2$
with  $N_1$ and $N_2$ even squares. There  is one cuspidal local system for
each unordered pair such that $N_1\neq N_2$ and two systems when $N_1=N_2$.
The  support is  the class  of $x$  where $C_\bG(x_s)^\circ$ is isogenous to
$\SO_{N_1}\times\SO_{N_2}$. If $N_1$ or $N_2$ is zero then $x_s=1$, otherwise
if  $N_1\neq  N_2$  then  $A_\bG(x_s)=\{1\}$  and  if  $N_1=N_2\neq  0$  then
$A_\bG(x_s)\simeq\BZ/2\BZ$  acting by the exchange  of the two $\SO$ factors.
The  Lusztig series is defined by $s$ such that $C_{\bG^*}(s)$ is
isogenous  to $\SO_{2a^2}\times\SO_{2b^2}$ where $a$ and $b$ are defined as
in  the $\PSO$  case; there  is a  unique such  class. If  $a$ and  $b$ are
distinct   (equivalently  $N_1$  and  $N_2$   non  zero),  the  centraliser
$C_{\bG^*}(s)$  is  connected.  If  $N_1=0$  (that is $a=b$) the group
$A_{\bG^*}(s)$ is $\BZ/2$, acting by exchanging the two components.

\subsubsection*{Case $\Spin_{2n}$}

There are cuspidal local systems coming from $\SO_{2n}$ or $\PSO_{2n}$ (see
Lemma  \ref{G tilde to G}) only if  $n$ is even. As above, let $2n=N_1+N_2$
where  $N_1$ and $N_2$ are even  squares. For each ordered pair $(N_1,N_2)$
with  nonzero $N_1$ and $N_2$, there is  a unique cuspidal local system. It
comes  from  $\PSO_{2n}$  if  $n\equiv  0\pmod  4$  and  from $\SO_{2n}$ if
$n\equiv  2\pmod 4$. If $N_1=0$  or if $N_2=0$ there  are two systems, each
with  the semisimple  part of  the support  central. The  lusztig series is
defined by $s$ such that
$C_{\bG^*}^\circ(s)\simeq\SO_{2a^2}\times^{\BZ/2\BZ}\SO_{2b^2}$, where
$N_1=(a+b)^2$  and $N_2=(a-b)^2$ with $a\in\BN  0$ and $b\in\BZ$. The group
$A_{\bG^*}(s)$  is trivial  if $N_1=N_2$,  is of  order 2 if $N_1$ and
$N_2$ are distinct nonzero and is of order 4 if $N_1$ or $N_2$ is zero.

We  now  look  at  systems  coming  from  $\frac12\Spin_{2n}$  and not from
$\PSO_{2n}$  (see Lemma  \ref{G tilde  to G}):  they exist  if $n\geq 6$ is
even,  of the form  $2n=N_1+N_2$ where $N_1$  and $N_2$ are even triangular
numbers. For each ordered pair $(N_1,N_2)$, if $N_1$ and $N_2$ are non zero
there  are two cuspidal  local system with  support the class  of $x$ such
that  $C_\bG(x_s)$ affords $\Spin_{N_1}\times\Spin_{N_2}$  as a double cover;
there  are two such classes if $N_1\neq N_2$ and only one if $N_1=N_2$. The
cases  $N_1=0$  and  $N_2=0$  correspond  to  4  cuspidal local systems with
semisimple  part of  the support  each one  of the  4 central elements. The
Lusztig  series is defined  by $s$ such  that, if $N_1=\frac{r(r+1)}2$
and      $N_2=\frac{t(t+1)}2$,      then     $C_{\bG^*}(s)^\circ\simeq
(\SO_{\frac{(r+t+1)^2}8}\times\Gl_{\frac{(r-t)^2-1}8}
\times\SO_{\frac{(r+t+1)^2}8})/\{\pm1\}$  or  $(\SO_{\frac{(r-t)^2}8}\times
\Gl_{\frac{(r+t)(r+t+2)}8}\times\SO_{\frac{(r-t)^2}8}) /\{\pm1\}$ depending
on  the values of $r$ and $t\pmod 4$ (as in the case of $\frac12\Spin$). If
$N_1\neq  N_2$ there is a single  such class; the group $A_{\bG^*}(s)$
is  noncyclic of order 4  generated by two elements  $x$ and $x'$ where $x$
acts  by the simultaneous  outer automorphism of  both $\SO$ components and
$x'$  acts  by  exchanging  the  two  $\SO$  components and doing the outer
automorphism  on  the  $\Gl$  component.  If  $N_1=N_2$, there are two such
classes;  the group $A_{\bG^*}(s)$  is of order  2 acting by the outer
automorphism of $\Gl$.

Finally  we look at  the cuspidal local  systems whose central character is
injective  on $Z\bG$ (see  \cite[23.2 (e)]{CS}): they  exist only if $n$ is
odd,  in which case $Z\bG$ has two  injective characters. For each of these
and  for each ordered pair $(N_1,N_2)$ of even triangular numbers such that
$2n=N_1+N_2$  there are two cuspidal local systems (note that $N_1\neq N_2$
and that $N_1$ and $N_2$ are nonconsecutive triangular numbers since $n$ is
odd).  The 4 complexes attached to $(N_1,N_2)$ and $(N_2,N_1)$ are obtained
from  each other by translating  by the various elements  of the centre. If
$N_1>0$  and $N_2>0$ the  two systems attached  to $(N_1,N_2)$ have support
the class of $x$ such that $C_\bG(x_s)$ affords
$\Spin_{N_1}\times\Spin_{N_2}$  as a  double cover.  If $N_1=0$  one of the
systems  has  unipotent  support,  the  other  as  well  as the two systems
parameterised by $(N_2,0)$ are deduced by translation by the centre.

The  Lusztig series of a cuspidal local system parameterised by $(N_1,N_2)$
is   defined   by   $s$   such   that,   if  $N_1=\frac{r(r+1)}2$  and
$N_2=\frac{t(t+1)}2$, then $C_{\bG^*}(s)^\circ\simeq
(\SO_{\frac{(r+t+1)^2}8}\times\Gl_{\frac{(r-t)^2-1}8}
\times\SO_{\frac{(r+t+1)^2}8})/\{\pm1\}$  or  $(\SO_{\frac{(r-t)^2}8}\times
\Gl_{\frac{(r+t)(r+t+2)}8}\times\SO_{\frac{(r-t)^2}8}) /\{\pm1\}$ depending
on  the  values  of  $r$  and  $t\pmod  4$  as  in  the  previous case (see
\cite[1.12]{remarks} or \cite[23.19 (e$_2$)]{CS}; note that $r\neq t$ since
$N_1\neq  N_2$  and  that  $|r-t|\neq  1  $  since  $N_1$ and $N_2$ are not
consecutive). We have $A_{\bG^*}(s)\simeq\BZ/4\BZ$ where the generator
acts  by  exchanging  the  two  $\SO$  components  and  twisting  the $\Gl$
components,  its square twisting the two $\SO$ components.

\section{Exceptional groups}\label{s:classex}
In  the following tables the labels for  the cuspidal local systems are, as
in  \cite{CS}, elements of  $\CM(\CG)$ for some  groups $\CG$. We fix
notations  for  the  labels  $(x,\chi)$  in  these  groups  as  follows: in
$\Sgot_n$  with  $n=2,3,4,5$,  we  let  $g_i$  for  $i=2,3,4,5$  denote  an
$i$-cycle;  $g'_2$ denotes the product  of two commuting transpositions and
$g_6$ is an element of order 6 of $\Sgot_5$. We denote by $\theta,\theta^2$
(resp.\  $i,-i$,  resp.\  $-\theta,-\theta^2$)  the  injective  characters of
$\BZ/3\BZ$    (resp.\   $\BZ/4\BZ$,   resp.\   $\BZ/6\BZ$).   Finally,   when
$C_\CG(x)$   is  a  Coxeter  group  (which  for  instance  happens  when
$\CG=\Sgot_4$  and $x\in\{1,g_2,g'_2\}$) we  denote by $\ve$ the
sign character of this group.

The  semi-simple part of the  support of the cuspidal  local system will be
specified  by giving  the isomorphism  type of  its centraliser  (if needed
stating  also  the  number  of  such  conjugacy  classes) or by an explicit
description.  In the case when this centraliser is a product of quasisimple
groups  with cyclic center, amalgamated by  part of their center, we denote
$z_i$  a generator of the center of  the $i$-th factor in order to describe
the amalgamation.

It  is proved in \cite{Lu12} that  all cuspidal local systems are ``clean''
in  the  sense  of  Lusztig.  In  bad  characteristic we could not find the
classification  in the litterature but the reader can check that our tables
are complete by using the knowledge of unipotently supported cuspidal local
systems, of isolated semi-simple elements and the argument
\cite[2.10.1]{ICC} of Lusztig.

The  parameterisation by  labels in  families of  character sheaves has been
worked  out in \cite{CS}, \cite{Shoji1} ar  \cite{Shoji2} except in a few
cases  marked by ``??''  in the tables.  We have completed this parameterisation
so that in every case we have:

\begin{property}
{\rm  (see \cite[6.2]{Shoji1},  \cite[4.6]{Shoji2} and  \cite{Os})} The
eigenvalue  of  Shintani  (``twisting''  operator  of  Shoji) on the local
system parameterised by $(x,\chi)$ is $\chi(x)/\chi(1)$.
\end{property}
In  the cases  marked ``??'',  it is  unknown whether the multiplicity property
given in \cite[17.8.3]{CS} holds.
\vfill\eject
\begin{sideways}
\begin{minipage}{\textheight}
\subsection*{Cuspidal local systems for type $G_2$}
The references are \cite[20.6]{CS} when the characteristic is good and
\cite[7.2--7.6]{Shoji1} when characteristic is  $2$ or $3$.

All  cuspidal  local  systems  are  in  the unipotent Lusztig series. their
labels  are in $\CM(\CG)$ with  $\CG=\Sgot_3$. The unipotent class of
$C_G(x_s)$ is always the subregular class of the whole group.

\hfill\break\vskip 1cm

\begin{tabular}{|l|c|c|c|c|c|c|}
\hline
Label & Conditions & $C_G(x_s)$ or $x_s$&
\begin{tabular}{c}Unipotent class\\in $C_G(x_s)$\end{tabular}
&$A_G(x)$& Local system & Eigenvalue of $\Sh$ \\
\hline
$(g_2,\ve)$&
\begin{tabular}{c}$p\neq 2$\\$p=2$\end{tabular}&
\begin{tabular}{c}$\SL_2\times^{\BZ/2\BZ} \SL_2$\\ $x_s=1$\end{tabular}
& $\reg$&$\BZ/2\BZ$&$\ve$&$-1$\\
\hline
$(g_3,\theta^i)$, $i=1,2$&
\begin{tabular}{c} $p\neq 3$\\$p=3$\end{tabular}
&
\begin{tabular}{c}$\SL_3$\\$x_s=1$\end{tabular}
&$\reg$&$\BZ/3\BZ$&$\theta^i$&$\theta^i$\\
\hline
$(1,\ve)$&
\begin{tabular}{c}$ p\neq 3$\\$p=3$\end{tabular}
&$x_s=1$& subregular&
\begin{tabular}{c}$\Sgot_3$\\$\BZ/2\BZ$\end{tabular}&$\ve$&$1$\\
\hline
\end{tabular}
\end{minipage}
\end{sideways}

\newcommand\stack[2]{{\genfrac{}{}{0pt}{0}{#1}{#2}}}

\begin{sideways}
\begin{minipage}{\textheight}
\subsection*{Cuspidal local systems for type $E_6^\ad$}
The  reference  is  \cite[20.3  (a)]{CS}.  When the characteristic is good,
cuspidal  local systems are in three Lusztig series which correspond to the
3   central  elements  $z^i$,  $i=0,1,2$,  of  the  dual  group.  When  the
characteristic  is  3,  all  cuspidal  local  systems  are in the unipotent
Lusztig  series. The labels are in $\CM(\CG)$ with $\CG=\Sgot_3$. The
$i$  in the table corresponds  to the series $z^i$.  The unipotent class of
$C_G(x_s)$  in good characteristic is the  class denoted by $D_4(A_1)$, whose
weighted diagram is $\begin{smallmatrix}
&&0&&\\0&0&2&0&0\end{smallmatrix}$.

\hfill\break\vskip 1cm
\newsavebox{\firsttablebox}
\begin{lrbox}{\firsttablebox}
\begin{tabular}{|l|c|c|c|c|c|c|}
\hline
Label & Conditions &$C_G(x_s)$ or $x_s$&
\begin{tabular}{c}Unipotent class\\in $C_G(x_s)$\end{tabular}
&$A_G(x)$& Local system&
\begin{tabular}{c}Eigenvalue\\of $\Sh$\end{tabular}
\\
\hline
$\stack{(g_3,\theta^j),}{j=1,2}$&
\begin{tabular}{c}$ p\neq 3$\\ \\$p=3$\end{tabular}&
\begin{tabular}{c}
$\dfrac{(\SL_3)^3}{z_1=z_2=z_3}\rtimes\BZ/3\BZ$\\($\BZ/3\BZ$ permutes
cyclically\\ the components)\\$x_s=1$\end{tabular}&
$\reg$&
\begin{tabular}{c}$\stack{(\BZ/3\BZ)^2}{=<x>\times\BZ/3\BZ}$\\ \\$\BZ/3\BZ$\end{tabular}&
\begin{tabular}{c}$\theta^j\boxtimes \theta^i$\\ \\$\theta^j$\end{tabular}
&$\theta^j$\\
\hline
\end{tabular}
\end{lrbox}
\scalebox{0.9}{\usebox{\firsttablebox}}
\end{minipage}
\end{sideways}

\begin{sideways}
\begin{minipage}{\textheight}
\subsection*{Cuspidal local systems for type $E_6^\sco$}
The  references are \cite[20.3  (b)]{CS} if $p\neq  2$ and \cite{Spalt} for
$p=2$. We assume $p\neq 3$, otherwise the classification is the same as for
the adjoint type.

\subsubsection*{Unipotent Lusztig series}

There  are two  cuspidal local  systems in  the unipotent  series. They are
lifted   from  the  adjoint  group;   the  labels  are  $(g_3,\theta)$  and
$(g_3,\theta^2)$,  (in $\CM(\CG)$ with  $\CG=\Sgot_3$). All the lifts
of  the $x$  of $E_6^\ad$  are conjugate.  Let $\tilde  x$ be  one of these
lifts;   we  have   $A_\bG(\tilde  x)\simeq(\BZ/3\BZ)^2=<\tilde  x>\times
Z_\bG$.

\subsubsection*{Non unipotent Lusztig series}
The  other cuspidal  systems are  in the  Lusztig series corresponding to a
semi-simple  element of the dual group  whose centraliser has its connected
component  of type $D_4$ and component  group $\BZ/3\BZ$; the labels are in
$\CM(\CG)$  with $\CG=\Omega\times\BZ/2\BZ$, where  $\BZ/2\BZ$ is the
group $\CG$ of a family of $D_4$ and $\Omega$ is a cyclic group of order
3. When the characteristic is good, the unipotent class is always the class
$A_5+A_1$  in the notation  of \cite{S}, which  is denoted by $E_6(a_3)$ in
\cite{C}; its weighted diagram is $\begin{smallmatrix}
&&0&&\\2&0&2&0&2\end{smallmatrix}$.  The  parameterisation  is unknown when
the  characteristic is 2. It is known up to swapping two lines in the table
if $p\neq 2$. The parameterisation that we give satisfies property ($*$) of
the beginning of this section.

\hfill\break\vskip 1cm
\newsavebox{\secondtablebox}
\begin{lrbox}{\secondtablebox}
\begin{tabular}{|l|c|c|c|c|c|c|}
\hline
Label
& Conditions & $C_G(x_s)$ or $x_s$&
\begin{tabular}{c}Unipotent class\\in $C_G(x_s)$\end{tabular}
&$A_G(x)$&
\begin{tabular}{c}Local\\system\end{tabular}&
\begin{tabular}{c}Eigenvalue\\of $\Sh$\end{tabular}\\
\hline
\begin{tabular}{c}
$(-\omega^j,\theta^i\boxtimes\ve)$ (??)\\
$j=1,2$;$i=0,1,2$\end{tabular}&
\begin{tabular}{c}$p\neq2$\\ \\$p=2$\end{tabular}&
\begin{tabular}{c}
$\SL_2\times^{\BZ/2\BZ}\SL_6$\\
($x_s=s_0z^i$ with $s_0$ of order 2)\\$x_s=z^i$\end{tabular}&
$\reg$ &$\stack{\BZ/6\BZ}{=<s_0z>}$&$-\theta^j$&$-\theta^{ij}$\\
\hline
\begin{tabular}{c}$(\omega^j,\theta^i\boxtimes\ve)$ (??)\\
$j=1,2$;$i=0,1,2$\\
\end{tabular}&
&
$x_s=z^i$&
\begin{tabular}{c}\cite{S}: $A_5+A_1$,\\ \cite{C}: $E_6(a_3)$\end{tabular}
&$\BZ/6\BZ$&$-\theta^j$&$\theta^{ij}$\\
\hline
\end{tabular}
\end{lrbox}
\scalebox{0.85}{\usebox{\secondtablebox}} 
\end{minipage}
\end{sideways}

\begin{sideways}
\begin{minipage}{\textheight}
\subsection*{Cuspidal local systems for type $E_7^\ad$.}
The reference is \cite[20.3 (c)]{CS}.

When the characteristic is good the cuspidal local systems are in 2 Lusztig
series  corresponding to the 2 central  elements $z^k$, $k=0,1$ of the dual
group.  When the characteristic is 2 they  are all in the unipotent Lusztig
series.  They are  in the  exceptional family  of the Weyl group $W(E_7)$.
When  the characteristic  is good  the unipotent  class is always the class
whose weighted diagram is
$\begin{smallmatrix}&&0&&&\\1&0&1&0&1&0\end{smallmatrix}$.  The $k$  in the
table corresponds to the series $z^k$.

\hfill\break\vskip 1cm
\newsavebox{\thirdtablebox}
\begin{lrbox}{\thirdtablebox}
\begin{tabular}{|l|c|c|c|c|c|c|}
\hline
Label & Conditions & $C_G(x_s)$ or $x_s$&
\begin{tabular}{c}Unipotent class \\in $C_G(x_s)$\end{tabular}
&$A_G(x)$&
\begin{tabular}{c}Local\\system\end{tabular}&
\begin{tabular}{c}Eigenvalue\\of $\Sh$\end{tabular}\\
\hline
$\stack{(-1,(-1)^j),}{j=0,1}$&
\begin{tabular}{c}$p\neq 2$\\ \\$p=2$\end{tabular}&
\begin{tabular}{c}
$\dfrac{\SL_4^2\times\SL_2}{z_1=z_2,\,z_1^2=z_3}\rtimes\BZ/2\BZ$
\\($\BZ/2\BZ$ swaps the two $\SL_4$)\\ $x_s=1$
\end{tabular}&
\begin{tabular}{c}$\reg\times\reg\times\reg$\\ \\$\reg$\end{tabular}&
\begin{tabular}{c}$\BZ/4\BZ\times\BZ/2\BZ$\\$=<x>\times\BZ/2\BZ$ \\$\BZ/4\BZ$\end{tabular}&
\begin{tabular}{c}$i^{2j+1}\boxtimes \ve^k$\\ \\
$i^{2j+1}$\end{tabular}&$(-1)^j$\\
\hline
\end{tabular}
\end{lrbox}
\scalebox{0.85}{\usebox{\thirdtablebox}} 
\end{minipage}
\end{sideways}

\begin{sideways}
\begin{minipage}{\textheight}
\subsection*{Cuspidal local systems for type $E_7^\sco$.}
We  assume $p\neq 2$, otherwise  the classification is the  same as in the
adjoint  group.  The  references  are  \cite[20.5]{CS}  for  $p\ne  3$  and
\cite{Spalt} for $p=3$.

\subsubsection*{Unipotent Lusztig series}
There  are 2 cuspidal local systems lifted from $E_7^\ad$; they are both in
the  unipotent Lusztig series. The two lifts  $\tilde x$ and $\tilde xz$ of
the $x$ of $E_7^\ad$ are conjugate. We have $A_\bG(\tilde
x)=\BZ/4\BZ\times\BZ/2\BZ=<\tilde x>\times Z_\bG$.

\subsubsection*{Non unipotent Lusztig series}
The other cuspidal local systems are in the Lusztig series corresponding to
an isolated semi-simple element of the dual group whose centraliser has its
connected  component of type $E_6$  and component group $\Omega$ isomorphic
to $\BZ/2\BZ$, acting as the outer automorphism of $E_6$. The labels of the
cuspidal local systems are in $\CM(\CG)$ where
$\CG=\Omega\times\Sgot_3$. When the characteristic is good the unipotent
class   is  always  $E_7(a_5)$  in  the  notation  of  \cite{C},  which  is
$u=D_6(a_2)+A_1$  in  the  notation  of  \cite{S};  its weighted diagram is
$\begin{smallmatrix}&&0&&&\\  0&0&2&0&0&2\end{smallmatrix}$.  We  denote by
$1$  and  $\omega$  the  elements  of  $\Omega$  and by $z$ the non-trivial
central element.

\hfill\break\vskip 1cm
\newsavebox{\fourthtablebox}
\begin{lrbox}{\fourthtablebox}
\begin{tabular}{|l|c|c|c|c|c|c|}
\hline
Label& Conditions & $C_G(x_s)$ or $x_s$&
\begin{tabular}{c}Unipotent class \\in $C_G(x_s)$\end{tabular}
&$A_G(x)$&
\begin{tabular}{c}Local\\system\end{tabular}&
\begin{tabular}{c}Eigenvalue\\of $\Sh$\end{tabular}\\
\hline
\begin{tabular}{c}$(\omega,\ve^i\boxtimes1),$\\$i=0,1$\\ (?? if $p=3$)
\end{tabular}&
&$x_s=z^i$&$D_6(a_2)+A_1$&$\Sgot_3\times\BZ/2\BZ$&
$\ve\boxtimes\ve$&$(-1)^i$\\
\hline
\begin{tabular}{c}$(\omega g_3,\ve^i\boxtimes \theta^j)$\\
$i=0,1$; $j=1,2$\\
(?? if $p=3$)
\end{tabular}&
\begin{tabular}{c}$p\neq 3$\\ \\$p=3$\end{tabular}&
\begin{tabular}{c}$\SL_6\times^{\BZ/3\BZ}\SL_3$\\
($x_s=s_0z^i$ with $s_0$ of order 3)\\
$x_s=z^i$\end{tabular}&
\begin{tabular}{c} $\reg\times\reg$\\ \\$\reg$\end{tabular}&
$\BZ/6\BZ$&$-\theta^j$&$(-1)^i\theta^j$\\
\hline
\end{tabular}
\end{lrbox}
\scalebox{0.85}{\usebox{\fourthtablebox}} 
\end{minipage}
\end{sideways}

\begin{sideways}
\begin{minipage}{\textheight}
\subsection*{Cuspidal local systems for type $E_8$}
The   references   are   \cite[21.12]{CS}   for   good  characteristic  and
\cite[section 5]{Shoji2} for bad characteristic. All  cuspidal local systems, for all  characteristics, are in the unipotent
Lusztig  series. The labels are  in $\CM(\Sgot_5)$. When the characteristic
is  good the unipotent class is, as observed by Lusztig in \cite[4.7
(a)]{Arcata} always the class $E_8(a_7)$ in the notation
of  \cite{C} ($2A_4$ in the notation  of \cite{S}).
\bigskip
\newsavebox{\fifthtablebox}
\begin{lrbox}{\fifthtablebox}
\begin{tabular}{|l|c|c|c|c|c|c|}
\hline
Label& Conditions & $C_G(x_s)$ or $x_s$& \begin{tabular}{c} Unipotent class\\
in $C_G(x_s)$\end{tabular}&
$A_G(x)$&
\begin{tabular}{c}Local\\system\end{tabular}&
\begin{tabular}{c}Eigenvalue\\of $\Sh$\end{tabular}\\
\hline
$(1,\ve)$& &$x_s=1$&
\begin{tabular}{c}\cite{S}: $2A_4$,\\\cite{C}: $E_8(a_7)$,\\
$\begin{smallmatrix}&&0&&&&\\0&0&0&2&0&0&0\end{smallmatrix}$\end{tabular}
&
$\Sgot_5$&$\ve$&$1$\\
\hline
$(g_2,-\ve)$&
\begin{tabular}{c}$p\neq 2$\\$p=2$\\ \\ \\\end{tabular}&
\begin{tabular}{c}$\SL_2\times E_7^\sco$\\$x_s=1$\\ \\ \\\end{tabular}&
\begin{tabular}{c}$\text{reg}\times(D_6(a_2)+A_1)$\\
{}\cite{S}: $E_7(a_2)+A_1$, \\
\cite{C}: $E_7(a_5)$\\
$\begin{smallmatrix}&&0&&&&\\0&0&2&0&0&2&2\end{smallmatrix}$\end{tabular}&
$\Sgot_3\times\BZ/2\BZ$&$\ve\boxtimes\ve$&$-1$\\
\hline
$(g_3,\ve\theta^j)$, $j=1,2$&
\begin{tabular}{c}$p\neq 3$\\$p=3$\end{tabular}&
$\SL_3\times E_6^\sco$&
\begin{tabular}{c}$\reg\times(A_5+A_1)$\\
$\reg\times\reg$
$\begin{smallmatrix}&&0&&&&\\0&0&2&0&2&2&2\end{smallmatrix}$\end{tabular}&
\begin{tabular}{c}$\BZ/6\BZ$\\$\BZ/3\BZ$\end{tabular}&
\begin{tabular}{c}$-\theta^j$\\$\theta^j$\end{tabular}&$\theta^j$\\
\hline
$(g_4,i^{2k+1})$, $k=0,1$&
\begin{tabular}{c}$p\neq 2$\\$p=2$\end{tabular}&
\begin{tabular}{c}
$\Spin_{10}\times^{\BZ/4\BZ}\SL_4$\\$x_s=1$
\end{tabular}&
\begin{tabular}{c}$(3,7)\times\reg$\\$E_8(a_1)$
$\begin{smallmatrix}&&2&&&&\\2&2&0&2&2&2&2\end{smallmatrix}$\end{tabular}&
$\BZ/4\BZ$&$i^{2k+1}$&$i^{2k+1}$\\
\hline
$(g_5,\zeta^j)$, $j=1,2,3,4$&
\begin{tabular}{c}$p\neq 5$\\$p=5$\end{tabular}&
\begin{tabular}{c}$\SL_5\times^{\BZ/5\BZ}\SL_5$\\$x_s=1$\end{tabular}&
$\reg$&$\BZ/5\BZ$&$\zeta^j$&$\zeta^j$\\
\hline
$(g_6,\ve\theta^j)$, $j=1,2$&
\begin{tabular}{c}$p\neq 2,3$\\ \\$p=2$\\$p=3$\\ \\ \\\end{tabular}&
\begin{tabular}{c}
$\dfrac{(\SL_2\times \SL_3)\times^{\BZ/6\BZ}\SL_6}
{z_3^2=z_2,\, z_3^3=z_1}$ \\
$\SL_2\times \SL_3\times^{\BZ/3\BZ}\SL_6$\\
$x_s=1$\\ \\ \\
\end{tabular}&
\begin{tabular}{c} $\reg$\\ \\
$\reg$\\{}\cite{S}: $ E_7+A_1$,\\ \cite{C}: $E_8(a_3)$\\
$\begin{smallmatrix}&&0&&&&\\2&0&2&0&2&2&2\end{smallmatrix}$ \end{tabular}&
$\BZ/6\BZ$&$-\theta^j$&$-\theta^j$\\
\hline
$(g'_2,\ve)$&
\begin{tabular}{c} $p\neq 2$\\$p=2$\\ \\ \\ \end{tabular}&
\begin{tabular}{c}
$\Spin_{16}$\\$x_s=1$\\ \\ \\ \end{tabular}&
\begin{tabular}{c} $ (1,3,5,7)$\\{}\cite{S}: $D_8(a_1)$, \\
\cite{C}: $E_8(a_5)$\\
$\begin{smallmatrix}&&0&&&&\\2&0&2&0&0&2&0\end{smallmatrix}$ \end{tabular}&
$W(B_2)$&
$\ve$&$1$\\
\hline
\end{tabular}
\end{lrbox}
\scalebox{0.85}{\usebox{\fifthtablebox}} 
\end{minipage}
\end{sideways}

\begin{sideways}
\begin{minipage}{\textheight}
\subsection*{Cuspidal local systems for type $F_4$}
The  references are \cite[21.3, 21.13]{CS}  when the characteristic is good
and \cite[7.2--7.6]{Shoji1} for $p=2,3$.

All  cuspidal local  systems are  in the  unipotent Lusztig  series. In all
characteristic the labels are in $\CM(\Sgot_4)$. When the characteristic is
good  the unipotent class of $C_\bG(x_s)$ is in $\bG$ the class $F_4(a_3)$.

\hfill\break\vskip 1cm
\newsavebox{\sixthtablebox}
\begin{lrbox}{\sixthtablebox}
\begin{tabular}{|l|c|c|c|c|c|c|}
\hline
Label& Conditions & $C_G(x_s)$ or $x_s$&  \begin{tabular}{c}Unipotent
class\\in $C_G(x_s)$\end{tabular}
&$A_G(x)$& Local system&
\begin{tabular}{c}Eigenvalue\\of $\Sh$\end{tabular}\\
\hline
$(1,\ve)$&
\begin{tabular}{c} $p\neq 2$\\$p=2$ \end{tabular}&
$x_s=1$&
$F_4(a_3)$,
$\begin{smallmatrix}0&2&0&0&\end{smallmatrix}$
&
\begin{tabular}{c}$\Sgot_4$ \\ $\Sgot_3$\end{tabular}&
$\ve$&$1$\\
\hline
$(g_2,\ve)$&
\begin{tabular}{c}$p\neq 2$\\$p=2$\end{tabular}&
\begin{tabular}{c}$ \Sp_6\times^{\BZ/2\BZ}\SL_2$\\$x_s=1$\end{tabular}&
\begin{tabular}{c}$(2,4)\times\reg$, $\begin{smallmatrix} 2&0&2\end{smallmatrix}$
\\subregular $F_4(a_1)$
$\begin{smallmatrix} 2&2&0&2\end{smallmatrix}$\end{tabular}& $\BZ/2\BZ$
& $\ve$ &$-1$ \\
\hline

$(g'_2,\ve)$&
\begin{tabular}{c} $p\neq 2$\\$p=2$ \end{tabular}&
\begin{tabular}{c}$ \Spin_9$\\$x_s=1$\end{tabular}&
\begin{tabular}{c}$(1,3,5)$, $\begin{smallmatrix} 0&2&0&2\end{smallmatrix}$
\\$F_4(a_2)$, $\begin{smallmatrix} 0&2&0&2\end{smallmatrix}$\end{tabular}&
$W(B_2)$&
\begin{tabular}{c}$\ve$, coming from $\SO_9 $\\
$\ve$\end{tabular}&$1$\\
\hline
$\stack{(g_3,\theta^j), }{j=1,2}$&
\begin{tabular}{c} $p\neq 3$\\$p=3$ \end{tabular}&
\begin{tabular}{c}$ \SL_3\times^{\BZ/3\BZ}\SL_3$\\$x_s=1$\end{tabular}&
\begin{tabular}{c}$\reg\times\reg$\\ $\reg$\end{tabular}&
$\BZ/3\BZ$&$\theta^j$&$\theta^j$\\
\hline
$\stack{(g_4,i^{2k+1}), }{k=0,1}$&
\begin{tabular}{c} $p\neq 2$\\$p=2$ \end{tabular}&
\begin{tabular}{c}$ \SL_4\times^{\BZ/2\BZ}\SL_2$\\$x_s=1$\end{tabular}&
\begin{tabular}{c}$\reg\times\reg$\\ $\reg$\end{tabular}&
$\BZ/4\BZ$&$i^{2k+1}$&$i^{2k+1}$\\
\hline
\end{tabular}
\end{lrbox}
\scalebox{0.9}{\usebox{\sixthtablebox}} 
\end{minipage}
\end{sideways}

\bigskip

\noindent Fran\c cois~Digne\\
Laboratoire Ami\'enois de Math\'ematique Fondamentale et Appliqu\'ee,
CNRS UMR 7352, Universit\'e de Picardie-Jules Verne,
80039 Amiens Cedex France.\\
E-mail: digne@u-picardie.fr\\

\medskip

\noindent Gustav~Lehrer\\
School of Mathematics and Statistics,
University of Sydney, NSW 2006, Australia.\\
E-mail: gustav.lehrer@sydney.edu.au\\

\medskip

\noindent Jean~Michel\\
Institut de Math\'ematiques de Jussieu -- Paris rive
gauche, Universit\'e Denis Diderot, B\^atiment Sophie Germain,
75013, Paris France.\\
E-mail: jean.michel@imj-prg.fr

\end{document}